\documentclass[twoside,11pt]{article}

\usepackage{jmlr2e}
\usepackage{amsmath,amssymb,amstext,bbm,algorithm,algorithmic,graphicx,wrapfig,
  color,multirow,bigstrut}
\RequirePackage[OT1]{fontenc}
\usepackage{stmaryrd,enumerate,latexsym,bm,amsfonts,
  subfigure,wrapfig,verbatim,cases}
\usepackage[small]{caption}

 \newcommand{\ceil}[1]{\lceil #1 \rceil}

 \newcommand{\sign}{\text{sign}}

\DeclareMathOperator{\sgn}{sgn}

\providecommand{\abs}[1]{\lvert#1\rvert}
\providecommand{\norm}[1]{\lVert#1\rVert}

\newcommand{\set}[1]{\left\{#1\right\}}

\newcommand{\mat}[1]{{\mathbf #1}}

\newcommand{\A}{\mat{A}}
\newcommand{\B}{\mat{B}}

\newenvironment{proof*}{\noindent{\bf Proof:}}{}
\newcommand{\ignore}[1]{}

\newcommand{\vzero}{\mathbf{0}}

\newcommand{\M}{\mat{M}}

\newcommand{\vI}{\mat{I}}

\newcommand{\re}[2]{\texttt{RE}(#1 \parallel #2)}
\newcommand{\dotp}[2]{\left \langle #1 , #2 \right \rangle}
\newcommand{\E}{\mathbb{E}}
\newcommand{\eps}{\varepsilon}
\newcommand{\vx}{\vec{x}}
\newcommand{\vy}{\vec{y}}
\newcommand{\ve}{\vec{e}}
\newcommand{\vz}{\vec{z}}

\newcommand{\vu}{\vec{u}}

\newcommand{\vp}{\vec{p}}

\newcommand{\vlams}{\bm{\lambda}^*}
\newcommand{\vlam}{\bm{\lambda}}

\newcommand{\vb}{\mat{b}}

\newcommand{\enc}[1]{\left( #1 \right)}
\newcommand{\enco}[1]{\left [ #1 \right ]}
\newcommand{\enct}[1]{\left \{ #1 \right \}}

\newcommand{\argmax}{\operatornamewithlimits{argmax}}
\newcommand{\argmin}{\operatornamewithlimits{argmin}}

\newcommand{\R}{\mathbb{R}}

\newcommand{\poly}{\mathrm{poly}}

\newcommand{\eqdef}{\stackrel{\vartriangle}{=}}

{\vspace{2ex} \begin{center} \textbf{\large Homework 4}
  \end{center}
  
  \begin{enumerate}}{\end{enumerate}}

\renewcommand{\vec}[1]{\mathbf{#1}}

\newcommand{\qed}{\hfill \small$\blacksquare$\normalsize}%{\hfill
\newcommand{\longv}[1]{#1}
\newcommand{\old}[1]{#1}

\ShortHeadings{The Rate of Convergence of AdaBoost}{Mukherjee, Rudin and Schapire}
\begin{document}

\title{The Rate of Convergence of AdaBoost}

\author{\name Indraneel Mukherjee \email imukherj@cs.princeton.edu\\
  \addr   Princeton University \\
  Department of Computer Science\\
  Princeton, NJ 08540 USA\\
  \AND
  \name Cynthia Rudin \email rudin@mit.edu \\
  \addr  Massachusetts Institute of Technology\\
  MIT Sloan School of Management\\
  Cambridge, MA 02139 USA\\
  \AND
  \name Robert E. Schapire \email schapire@cs.princeton.edu\\
  \addr   Princeton University \\
  Department of Computer Science\\
  Princeton, NJ 08540 USA}

\editor{}

\maketitle

\begin{abstract}%
  \old{
 The AdaBoost algorithm was designed to combine
  many ``weak'' hypotheses that perform slightly better than random
  guessing into a ``strong'' hypothesis that has very low error.  We
  study the rate at which AdaBoost iteratively converges to the
  minimum of the ``exponential loss.''
  Unlike previous work,
  our proofs do not require a weak-learning assumption,
  nor do they require that minimizers of the exponential loss are
  finite.  Our first result shows that at iteration $t$,
  the exponential loss of AdaBoost's computed parameter vector will be
  at most $\varepsilon$ more than that of any parameter vector of
  $\ell_1$-norm bounded by $B$ in a number of rounds that is at most
  a polynomial in $B$ and $1/\varepsilon$. We also provide
  lower bounds showing that a polynomial dependence on these
  parameters is necessary.  Our second result is that within
  $C/\varepsilon$ iterations, AdaBoost achieves a value of the
  exponential loss that is at most $\varepsilon$ more than the best
  possible value, where $C$ depends on the dataset. We show that this
  dependence of the rate on $\varepsilon$ is optimal up to constant
  factors, i.e., at least $\Omega(1/\varepsilon)$ rounds are necessary
  to achieve within $\varepsilon$ of the optimal exponential loss.
}
\end{abstract}

\begin{keywords}
  AdaBoost, optimization, coordinate descent, convergence rate.
\end{keywords}

%\newcolumntype{C}{>{\centering}p{12mm}@{}}

%%%%%%%%%% Cynthia's macros %%%%%%%%%%%%

\newcommand{\brackets}[1]{\left[{#1}\right]}

\newcommand{\exploss}{{L}}
\renewcommand{\poly}{{\rm poly}}
\newcommand{\elone}[1]{{\|{#1}\|_1}}
\def\llambda{\bm{\lambda}}
\def\numhyps{{N}}
\def\X{\mathcal{X}}
\def\hyp{\mathcal{H}}
\def\th{^\textrm{th}}
\def\dplus{\sum_{i:M_{ij}=1} D_{t}(i) }
\def\dminus{\sum_{i:M_{ij}=-1} D_{t}(i) }
\newcommand{\bh}{\hbar} 
\def\sumim{\sum_{i=1}^m}
\newcommand{\lambdaexp}[1]{{\llambda^{{#1}}}}
\newcommand{\reals}{\mathbb{R}}
\newcommand{\nolineskips}{%
\setlength{\parskip}{0pt}%
\setlength{\parsep}{0pt}%
\setlength{\topsep}{0pt}%
\setlength{\partopsep}{0pt}%
\setlength{\itemsep}{0pt}}
\newcommand{\pr}[2]{{\rm Pr}_{#1}\brackets{#2}} % Probability of event #2 according to distribution #1
\newcommand{\sfrac}[2]{\mbox{$\frac{#1}{#2}$}}
\newcommand{\paren}[1]{\left({#1}\right)}
\newcommand{\hf}{H}  % final classifier

%%%%%%%%%% My macros %%%%%%%%%%%%

\newcommand{\vtau}{\val^\tau_{(t)}}
\newcommand{\vnew}{\tilde{\val}_{(t)}}
\newcommand{\vtnew}{\val^{\tau+1}_{(t)}}
\newcommand{\Abs}{AdaBoost.\emph{S} }
\newcommand{\phis}{\phi^*}
\newcommand{\loss}[2]{\ell^{#1}( #2 )}
\newcommand{\nloss}[3]{\bar{\ell}_{#3}^{#1}( #2 )}
\newcommand{\dvlam}{\Delta\bm{\lambda}}
\newcommand{\vlamd}{\bm{\lambda}^{\dagger}}
\newcommand{\val}{\bm{\eta}}
\newcommand{\vals}{\bm{\eta}^*}
\newcommand{\lam}{\lambda}
\newcommand{\ds}{\Delta S}
\newcommand{\dr}{\Delta R}
\newcommand{\dz}{\Delta\zeta}
\newcommand{\dal}{\Delta\val}
\newcommand{\alf}{\val^\dagger}
\newcommand{\dth}{\Delta\theta}
\newcommand{\dtho}{\Delta\theta_1}
\newcommand{\dtht}{\Delta\theta_2}
\newcommand{\led}[2]{\delta_X(#1; #2)}
\newcommand{\leds}[3]{\delta_{#3}(#1; #2)}
\newcommand{\mal}{\widehat{(\M_F\val)}}
\newcommand{\lmin}{\lambda_{\min}}
\newcommand{\nval}{\widehat{\val}}
\renewcommand{\oval}{\tilde{\val}}
\newcommand{\dY}{\Delta Y}
\newcommand{\dy}{\Delta y}
\newcommand{\lf}{L}
\newcommand{\grad}{\nabla}
\newcommand{\Yt}{\tilde{Y}}
\newcommand{\ddz}{\Delta z}
\renewcommand{\abs}[1]{\left | #1 \right |}
\newcommand{\nlam}{\norm{\vlams}_1}
\newcommand{\zl}{Z}
\newcommand{\fl}{F}

\old{
\section{Introduction}

The AdaBoost algorithm of \citet{FreundSc97} was designed to combine
many ``weak'' hypotheses that perform slightly better than random
guessing into a ``strong'' hypothesis that has very low error.
Despite extensive theoretical and empirical study,
basic properties of AdaBoost's convergence are not fully
understood. In this work, we focus on one of those properties, namely,
to find convergence rates that hold in the absence of any simplifying
assumptions.
Such assumptions, relied upon in much of the preceding work, make it easier
to prove a fast convergence rate for AdaBoost, but often do not hold
in the cases
where AdaBoost is commonly applied.

AdaBoost can be viewed as a coordinate descent (or functional gradient
descent) algorithm that iteratively minimizes an objective function
$\exploss:\reals^n\rightarrow \reals$ called the \textit{exponential
  loss} \citep{Breiman99,FreanDo98,FriedmanHaTi00,%
  Friedman01,%
  MasonBaBaFr00,%
  OnodaRaMu98,%
  RatschOnMu01,%
  SchapireSi99%
}.
Given $m$ labeled training
examples $(x_1,y_1),\ldots,(x_m,y_m)$, where the $x_i$'s are in some
domain $\X$ and $y_i\in\{-1,+1\}$, and a finite (but typically very
large) space of weak hypotheses
$\hyp=\{\bh_1,\ldots,\bh_{\numhyps}\}$, where each
$\hbar_j:\X\rightarrow \{-1,+1\}$, the exponential loss
is defined as
\[ \exploss(\llambda) \eqdef \frac{1}{m}
\sumim \exp\left(-\sum_{j=1}^{\numhyps} \lambda_j y_i
  \bh_j(x_i)\right)
\]
where $\llambda=\langle \lambda_1,\ldots,\lambda_N \rangle$
is a vector of weights or parameters.
In each iteration, a coordinate descent algorithm moves some distance
along some coordinate direction $\lambda_j$.
For AdaBoost, the coordinate
directions correspond to the individual weak
hypotheses. Thus, on each round, AdaBoost chooses some weak hypothesis and
step length, and adds these to the current weighted combination of
weak hypotheses, which is equivalent to updating a single weight. The
direction and step length are so chosen that the resulting vector
$\llambda^t$ in iteration $t$ yields a lower value of the exponential
loss than in the previous iteration,
$\exploss(\llambda^t)<\exploss(\llambda^{t-1})$. This repeats until it
reaches a minimizer if one exists. It was shown by
\citet{CollinsScSi02}, and later by \citet{ZhangYu05}, that AdaBoost
asymptotically converges to the minimum possible exponential loss.
That is,
\[  \lim_{t\rightarrow\infty} \exploss(\lambdaexp{t})
            = \inf_{\llambda\in\reals^\numhyps} \exploss(\llambda).
\]
However, that work did not address a convergence rate to the minimizer
of the exponential loss.

Our work specifically addresses a recent conjecture of
\citet{Schapire10} stating that there exists a positive constant $c$
and a polynomial $\poly()$ such that for all training sets and all
finite sets of weak hypotheses, and for all $B>0$,
\begin{equation}
  \label{rate:schap_conj:eqn}
   \exploss(\lambdaexp{t}) \leq
     \min_{\llambda: \elone{\llambda}\leq B} \exploss(\llambda)
       +       {\frac{\poly(\log {\numhyps},m,B)}
                                   {t^c}
               }.
\end{equation}
In other words, the exponential loss of AdaBoost will be at most
$\varepsilon$ more than that of any other parameter vector $\llambda$
of $\ell_1$-norm bounded by $B$ in a number of rounds that is bounded
by a polynomial in $\log \numhyps$, $m$, $B$ and $1/\varepsilon$.  (We
require $\log \numhyps$ rather than $\numhyps$ since the number of
weak hypotheses will typically be extremely large.) Along with an
upper bound that is polynomial in these parameters, we also provide
lower bound constructions showing some polynomial dependence on
$B$ and $1/\eps$ is necessary. Without any additional assumptions on the
exponential loss $\exploss$, and without altering AdaBoost's
minimization algorithm for $\exploss$, the best known convergence rate
of AdaBoost prior to this work that we are aware of is that of
\citet{BickelRiZa06} who prove a bound on the rate of the form
$O(1/\sqrt{\log t})$.

We provide also a convergence rate of AdaBoost to the minimum value of
the exponential loss. Namely, within $C/\epsilon$ iterations, AdaBoost
achieves a value of the exponential loss that is at most $\epsilon$
more than the best possible value, where $C$ depends on the dataset.
This convergence rate is different from the one discussed above in
that it has better dependence on $\epsilon$ (in fact the dependence is
optimal, as we show), and does not depend on the best solution within
a ball of size $B$. However, this second convergence rate cannot be
used to prove \eqref{rate:schap_conj:eqn} since in
certain worst case situations, we show the constant $C$ may be larger
than $2^m$ (although usually it will be much smaller).

Within the proof of the second convergence rate, we provide a lemma
(called the \textit{decomposition lemma}) that shows that the training
set can be split into two sets of examples: the ``finite margin
set,'' and the ``zero loss set.'' Examples in the finite margin set
always make a positive contribution to the exponential loss, and they
never lie too far from the decision boundary. Examples in the zero
loss set do not have these properties. If we consider the exponential
loss where the sum is only over the finite margin set (rather than
over all training examples), it is minimized by a finite
$\llambda$. The fact that the training set can be decomposed into
these two classes is the key step in proving the second convergence
rate.

This problem of determining the rate of convergence is relevant in the
proof of the consistency of AdaBoost given by \citet{BartlettTr07},
where it has a direct impact on the rate at which AdaBoost converges
to the Bayes optimal classifier (under suitable assumptions). It may
also be relevant to practitioners who wish to have a guarantee on the
exponential loss value at iteration $t$
(although, in general, minimization of the exponential loss need not
be perfectly correlated with test accuracy).

There have been several works that make additional assumptions on the
exponential loss in order to attain a better bound on the rate, but
those assumptions are not true in general, and cases are known where
each of these assumptions are violated. For instance, better bounds
are proved by \citet{RatschMiWa02} using results from \citet{LuoTs92},
but these appear to require that the exponential loss be minimized by
a finite $\llambda$, and also depend on quantities that are not easily
measured. There are many cases where $\exploss$ does not have a finite
minimizer; in fact, one such case is provided by \citet{Schapire10}.
\citet{ShalevshwartzSi08} have proven bounds for a variant of
AdaBoost. \citet{ZhangYu05} also have given rates of convergence, but
their technique requires a bound on the change in the size of
$\llambda^t$ at each iteration that does not necessarily hold for
AdaBoost.  Many classic results are known on the convergence of
iterative algorithms generally \citep[see for
instance][]{LuenbergerYe08,BoydVa04}; however, these typically start
by assuming that the minimum is attained at some finite point in the
(usually compact) space of interest, assumptions that do not generally
hold in our setting.
When the weak learning assumption
holds, there is a parameter $\gamma>0$ that governs the improvement of
the exponential loss at each iteration. \citet{FreundSc97} and
\citet{SchapireSi99} showed that the exponential loss is at most
$e^{-2t\gamma^2}$ after $t$ rounds, so AdaBoost rapidly converges to
the minimum possible loss under this assumption.

In Section \ref{rate:sec:adab} we summarize the coordinate descent
view of AdaBoost. Section \ref{rate:B:sec} contains the proof of the
conjecture, with associated lower bounds proved in
Section~\ref{rate:B_lbnd:sec}.
Section~\ref{rate:eps:sec}
provides the $C/\epsilon$ convergence rate.
The proof of the decomposition lemma is given in
Section~\ref{rate:dec:sec}.

%%%%%%%%%%%%%%%%%%%%%% BEGIN
%%%%%%%%%%%%%%%%%%%%%% FIGURE %%%%%%%%%%%%%%%%%%%%%%%%%%%%%%%%%%%

\section{Coordinate Descent View of AdaBoost}
\label{rate:sec:adab}
From the examples $(x_1,y_1),\ldots,(x_m,y_m)$ and hypotheses
$\hyp=\{\bh_1,\ldots,\bh_{\numhyps}\}$, AdaBoost iteratively computes
the function $F:\X\rightarrow \mathbb{R}$, where $\sign(F(x))$ can be
used as a classifier for a new instance $x$. The function $F$ is a
linear combination of the hypotheses. At each iteration $t$, AdaBoost
chooses one of the weak hypotheses $h_t$ from the set $\hyp$, and
adjusts its coefficient by a specified value $\alpha_t$. Then $F$ is
constructed after $T$ iterations as: $F(x) = \sum_{t=1}^T \alpha_t
h_t(x)$.  Figure \ref{rate:fig:adaboost} shows the AdaBoost algorithm
\citep{FreundSc97}.
\begin{figure}[t]
\hrule
\vspace{.3em}
{%\pseudocodefont
%\begin{tabbing}
\makebox[.4in][l]{Given:}
   $(x_1,y_1),\ldots,(x_m,y_m)$ where $x_i\in \X$, $y_i\in\{-1,+1\}$
\\
\makebox[.4in]{~}
   set $\hyp=\{\bh_1,\ldots,\bh_{\numhyps}\}$ of weak hypotheses
            $\bh_j:\X\rightarrow\{-1,+1\}$.
%\end{tabbing}
            \\
Initialize: $D_1(i) = 1/m$ for $i=1,\ldots,m$.\\
For $t = 1,\ldots, T$:
\vspace{-.8em}
\begin{itemize}
\nolineskips
\item Train weak learner using distribution $D_t$; that is, find weak
  hypothesis $h_t \in \hyp$ whose correlation $r_t \eqdef \E_{i\sim
    D_t}\enco{y_ih_t(x_i)}$ has maximum magnitude $\abs{r_t}$.
\item Choose $\displaystyle \alpha_t = \sfrac{1}{2}
  \ln\enct{\enc{1+r_t}/\enc{1-r_t}}$.
\item
 Update,
 for $i=1,\ldots,m$:
$\displaystyle
 D_{t+1}(i)
 = 
{D_t(i) \exp(-\alpha_t y_i h_t(x_i))}/
		     {Z_t}
$\\
where $Z_t$ is a normalization factor (chosen so that $D_{t+1}$
will be a distribution).
\end{itemize}
\vspace{-.8em}
Output the final hypothesis:
$ %\displaystyle
  F(x) = \sign\paren{\sum_{t=1}^T \alpha_t h_t(x)}
$.
}
\caption{The boosting algorithm AdaBoost.}
\label{rate:fig:adaboost}
\vspace{.3em}
\hrule
\vspace{-.7em}
\end{figure}

Since each $h_t$ is equal to $\bh_{j_t}$ for some $j_t$, $F$ can also
be written $F(x)= \sum_{j=1}^{\numhyps} \lambda_j \bh_j(x)$ for a
vector of values $\llambda=\langle
\lambda_1,\ldots\lambda_{\numhyps}\rangle$ (such vectors will
sometimes also be referred to as \emph{combinations}, since they represent
combinations of weak hypotheses). In different notation, we can write
AdaBoost as a coordinate descent algorithm on vector $\llambda$. We
define the \emph{feature matrix} $\M$ elementwise by $M_{ij}=y_i
\bh_j(x_i)$, so that this matrix contains all of the inputs to
AdaBoost (the training examples and hypotheses). Then the exponential
loss can be written more compactly as:
\[
L(\llambda) = \frac{1}{m}\sum_{i=1}^m e^{-(\M\llambda)_i}
\]
where $(\M\llambda)_i$, the $i\th$ coordinate of the vector
$\M\llambda$, is the (unnormalized) \emph{margin} achieved by vector
$\llambda$ on training example $i$.

Coordinate descent algorithms choose a coordinate at each iteration
where the directional derivative is the steepest, and choose a step
that maximally decreases the objective along that coordinate. To
perform coordinate descent on the exponential loss, we determine the
coordinate $j_t$ at iteration $t$ as follows, where $\mathbf{e}_j$ is
a vector that is 1 in the $j\th$ position and 0 elsewhere:
\begin{eqnarray}\label{rate:jteqn}
  j_t&\in& \argmax_j \abs{\left(- \frac{dL(\llambda^{t-1}+\alpha
        \mathbf{e}_j)}{d\alpha} \Big|_{\alpha=0}\right)} =
  \argmax_j \frac{1}{m}\abs{\sum_{i=1}^m e^{-(\M\llambda^{t-1})_i}M_{ij}}. 
\end{eqnarray}
% It can be shown \citep[see for instance][]{MasonBaBaFr00} that the
% distribution $D_t$ chosen by AdaBoost at each round $t$ puts weight
% $D_t(i)$ proportional to $e^{(-\M\vlam^{t-1})_i}$.
We can show that this is equivalent to the weak learning step of
AdaBoost. Unraveling the recursion in Figure \ref{rate:fig:adaboost}
for AdaBoost's weight vector $D_t$, we can see that $D_t(i)$ is
proportional to
\[
\exp\left(-\sum_{t'< t} \alpha_{t'} y_i h_{t'}(x_i)\right).
\]
The term in the exponent can also be rewritten in terms of the vector
$\llambda^t$, where $\lambda^t_j$ is the sum of $\alpha_t$'s where
hypothesis $\bh_j$ was chosen: $\sum_{t'< t} \alpha_{t'}
\mathbf{1}_{[\bh_j=h_{t'}]} = \lambda_{t-1,j}$. The term in the exponent
is:
\[
\sum_{t'< t} \alpha_{t'} y_i h_{t'}(x_i)=
\sum_j\sum_{t'< t}\alpha_{t'}
\mathbf{1}_{[\bh_j=h_{t'}]}y_i\bh_j(x_i)=\sum_j
\lambda^{t-1}_jM_{ij}=(\M\llambda^{t-1})_i,
\]
where $(\cdot)_i$ denotes the $i$th component of a vector. This means
$D_t(i)$ is proportional to $e^{-(\M\llambda^{t-1})_i}$.
Eq. \eqref{rate:jteqn} can now be rewritten as
\[
j_t\in \argmax_j \abs{\sum_i D_{t}(i) M_{ij}} =
\argmax_j\Big | \E_{i\sim D_t}\enco{M_{ij}} \Big | =
\argmax_j\Big | \E_{i\sim D_t}\enco{y_ih_j(x_i)} \Big |,
\]
which is exactly the way AdaBoost chooses a weak hypothesis in each
round (see Figure~\ref{rate:fig:adaboost}).  The correlation
$\sum_iD_{t}(i)M_{ij_t}$ will be denoted by $r_t$ and its absolute
value $\abs{r_t}$ denoted by $\delta_t$. The quantity $\delta_t$ is
commonly called the \textit{edge} for round $t$. The distance
$\alpha_t$ to travel along direction $j_t$ is found for coordinate
descent via a linesearch  \citep[see for
instance][]{MasonBaBaFr00}:
\begin{equation*}
  0=-\frac{dL(\llambda_t+\alpha_t \mathbf{e}_{j_t})}{d\alpha_t} =
  \sum_i e^{-\left(M(\llambda_t + \alpha_t
      \mathbf{e}_{j_t})\right)_i}M_{ij_t} 
\end{equation*}
and dividing both sides by the normalization factor,
\begin{equation*}
0= \dplus e^{-\alpha_t} - \dminus e^{\alpha_t}  = (1+r_t)
e^{-\alpha_t} - (1-r_t)e^{\alpha_t} \implies \alpha_t =
\frac{1}{2}\ln\enc{\frac{1+r_t}{1-r_t}}, 
\end{equation*}
just as in Figure \ref{rate:fig:adaboost}. Thus, AdaBoost is equivalent to
coordinate descent on $L(\llambda)$. With this choice of step length, it
can be shown \citep{FreundSc97} that the
exponential loss drops by an amount depending on the edge:
\begin{eqnarray*}
  L(\llambda_{t}) &=& L\left(\llambda_{t-1} +\alpha_t
    \mathbf{e_{j_t}}\right)=\left(\dplus e^{-\alpha_t} + \dminus
    e^{\alpha_t}\right)L(\llambda_{t-1})\\ 
  &=&\left((1+r_t) e^{-\alpha_t} + (1-r_t)
    e^{\alpha_t}\right)L(\llambda_{t-1}) =
  \left(2\sqrt{(1+r_t)(1-r_t)} \right)L(\llambda_t)\\
  &=& \left(\sqrt{1-r_t^2}\right) L(\llambda_{t-1})
  = \left(\sqrt{1-\delta_t^2}\right) L(\llambda_{t-1}). 
\end{eqnarray*}
Our rate bounds also hold when the weak-hypotheses are
confidence-rated, that is, giving real-valued predictions in
$[-1,+1]$, so that $h:\X\to[-1,+1]$. In that case, the criterion for
picking a weak hypothesis in each round remains the same, that is, at
round $t$, an $\bh_{j_t}$ maximizing the absolute correlation $ j_t\in
\textrm{argmax}_j \abs{\sum_{i=1}^m e^{-(\M\llambda^{t-1})_i}M_{ij}}$,
is chosen, where $M_{ij}$ may now be non-integral. An exact analytical
line search is no longer possible, but if the step size is chosen in
the same way,
\begin{equation}
  \label{rate:step:eqn}
\alpha_t = \frac{1}{2}\ln\enc{\frac{1+r_t}{1-r_t}},
\end{equation}
then \citet{FreundSc97} and \citet{SchapireSi99} show that a similar
drop in the loss is still guaranteed:
\begin{equation}
  \label{rate:drop:eqn}
  \lf(\vlam^{t}) \leq \lf(\vlam^{t-1})\sqrt{1-\delta_t^2}.
\end{equation}
With confidence rated hypotheses, other implementations may choose the
step size in a different way. However, in this paper, by ``AdaBoost'' we
will always mean the version in \citep{FreundSc97, SchapireSi99} which
chooses step sizes as in \eqref{rate:step:eqn}, and enjoys the loss
guarantee as in \eqref{rate:drop:eqn}.  That said, all our proofs work
more generally, and are robust to numerical inaccuracies in the
implementation. In other words, even if the previous conditions are
violated by a small amount, similar bounds continue to hold, although
we leave out explicit proofs of this fact to simplify the
presentation.
}
\longv{
\section{First convergence rate: Convergence to any target loss}
\label{rate:B:sec}

In this section, we bound the number of rounds of AdaBoost required to
get within $\eps$ of the loss attained by a parameter vector $\vlams$
as a function of $\eps$ and the $\ell_1$-norm $\nlam$. The vector
$\vlams$ serves as a reference based on which we define the target
loss $\lf(\vlams)$, and we will show that its $\ell_1$-norm measures
the difficulty of attaining the target loss in a specific sense. We prove a
bound polynomial in $1/\eps$, $\nlam$ and the number of examples $m$,
showing \eqref{rate:schap_conj:eqn} holds, thereby resolving
affirmatively the open problem posed in \citep{Schapire10}. Later in
the section we provide lower bounds showing how a polynomial
dependence on both parameters is necessary.

\subsection{Upper Bound}
\label{rate:ub:sec}
The main result of this section is the following rate upper bound.
\begin{theorem}
  \label{rate:B_rate:thm}
  For any $\vlams\in\R^N$, AdaBoost achieves loss at most
  $\lf(\vlams)+\eps$ in at most $13\nlam^6\eps^{-5}$ rounds. 
\end{theorem}
}
\old{
The high level idea behind the proof of the theorem is as follows. To
show a fast rate, we require a large edge in each round, as indicated
by \eqref{rate:drop:eqn}. A large edge is guaranteed if the size of
the current solution of AdaBoost is small. Therefore AdaBoost makes
good progress if the size of its solution does not grow too fast. On
the other hand, the increase in size of its solution is given by the
step length, which in turn is proportional to the edge achieved in
that round. Therefore, if the solution size grows fast, the loss also
drops fast. Either way the algorithm makes good progress. In the rest
of the section we make these ideas concrete through a sequence of
lemmas.

We provide some more notation. Throughout, $\vlams$ is fixed, and its
$\ell_1$-norm is denoted by $B$ \citep[matching the notation
in][]{Schapire10}. One key parameter is the suboptimality $R_t$ of
AdaBoost's solution measured via the logarithm of the exponential
loss:
\[
R_t \eqdef \ln \lf(\vlam^t) - \ln\lf(\vlams).
\]
Another key parameter is the $\ell_1$-distance $S_t$ of AdaBoost's
solution from the closest combination that achieves the target loss:
\[
S_t \eqdef \inf_{\vlam}\set{\norm{\vlam - \vlam^t}_1: \lf(\vlam) \leq
  \lf(\vlams)}.
\]
We will also be interested in how they change as captured by
\begin{eqnarray*}
  \dr_t \eqdef R_{t-1} - R_t, &&
  \ds_t \eqdef S_t - S_{t-1}.
\end{eqnarray*}
Notice that $\dr_t$ is always non-negative since AdaBoost decreases
the loss, and hence the suboptimality, in each round. Let $T_0$ be the
bound on the number of rounds in Theorem~\ref{rate:B_rate:thm}. We
assume without loss of generality that $R_0,\ldots,R_{T_0}$ and
$S_0,\ldots,S_{T_0}$ are all strictly positive, since otherwise the
theorem holds trivially. Also, in the rest of the section, we restrict
our attention entirely to the first $T_0$ rounds of boosting. We first
show that a $\poly(B,\eps^{-1})$ rate of convergence follows if the
edge is always polynomially large compared to the suboptimality.
\begin{lemma}
  \label{rate:edge_B_rate:lem}
  If for some constants $c_1,c_2$, where $c_2>1/2$, the edge satisfies
  $\delta_t \geq B^{-c_1}R_{t-1}^{c_2}$ in each round $t$, then
  AdaBoost achieves at most $L(\vlams) + \eps$ loss after
  $2B^{2c_1}(\eps\ln 2)^{1-2c_2}$ rounds.
\end{lemma}
\begin{proof}
  From the definition of $R_t$ and \eqref{rate:drop:eqn} we have
  \begin{equation}
    \label{rate:dr:eqn}
  \dr_t = \ln\lf(\vlam^{t-1}) - \ln\lf(\vlam^{t}) \geq
  -\frac{1}{2}\ln(1-\delta_t^2).
  \end{equation}
  Combining the above with the inequality $e^x \geq 1+x$, and the
  assumption on the edge
  \[
  \dr_t \geq  -\frac{1}{2}\ln(1-\delta_t^2)
  \geq \frac{1}{2}\delta_t^2 \geq \frac{1}{2}B^{-2c_1}R_{t-1}^{2c_2}.
  \]
  Let $T=\ceil{2B^{2c_1}(\eps\ln 2)^{1-2c_2}}$ be the bound on the
  number of rounds in the lemma. If any of $R_0,\ldots,R_T$ is
  negative, then by monotonicity $R_T < 0$ and we are done. Otherwise,
  they are all non-negative. Then, applying Lemma~\ref{rate:rec:lem}
  from the Appendix to the sequence $R_0,\ldots,R_T$, and using
  $c_2>1/2$ we get
  \[
  R_T^{1-2c_2} \geq R_0^{1-2c_2} + c_2B^{-2c_1}T > (1/2)B^{-2c_1}T \geq
  (\eps\ln 2)^{1-2c_2} \implies R_T < \eps\ln 2.
  \]
  If either $\eps$ or $\lf(\vlams)$ is greater than 1, then the lemma
  follows since $\lf(\vlam^T) \leq \lf(\vlam^0) = 1 < \lf(\vlams) +
  \eps$. Otherwise, 
  \[
  \lf(\vlam^T) < \lf(\vlams) e^{\eps\ln 2}
  \leq \lf(\vlams)(1+\eps)
  \leq \lf(\vlams) + \eps,
  \]
  where the second inequality uses $e^x \leq 1 + (1/\ln 2)x$ for $x\in
  [0,\ln 2]$.
\end{proof}
We next show that large edges are achieved provided $S_t$ is small
compared to $R_t$.
\begin{lemma}
  \label{rate:B_edge:lem}
  In each round $t$, the edge satisfies $\delta_t \geq
  R_{t-1}/S_{t-1}$. 
\end{lemma}
\begin{proof}
  For any combination $\vlam$, define $p_{\vlam}$ as the distribution
  on examples $\set{1,\ldots,m}$ that puts weight proportional to the
  loss $D_{\vlam}(i) = e^{-(\M\vlam)_i}/(m\lf(\vlam))$. Choose any
  $\vlam$ suffering at most the target loss $\lf(\vlam) \leq
  \lf(\vlams)$. By non-negativity of relative entropy we get
  \begin{eqnarray}
    0 &\leq& \re{D_{\vlam^{t-1}}}{D_{\vlam}} 
    = \sum_{i=1}^m D_{\vlam^{t-1}}\ln
    \enc{\frac{\frac{1}{m}e^{-(\M\vlam^{t-1})_i}/\lf(\vlam^{t-1})}
      {\frac{1}{m}e^{-(\M\vlam)_i}/\lf(\vlam)}} \nonumber \\
    \label{rate:re:eqn}
    &=& -R_{t-1} +
  \sum_{i=1}^mD_{\vlam^{t-1}}(i)\enc{\M\vlam - \M\vlam^{t-1}}_i.
  \end{eqnarray}
  Note that $D_{\vlam^{t-1}}$ is the distribution $D_t$ that AdaBoost
  creates in round $t$. The above summation can be rewritten as
  \begin{eqnarray}
    \sum_{i=1}^mD_{\vlam^{t-1}}(i)
      \sum_{j=1}^N\enc{\lam_j-\lam_j^{t-1}}M_{ij}
      &=& 
      \sum_{j=1}^N\enc{\lam_j-\lam_j^{t-1}}\sum_{i=1}^mD_t(i)M_{ij}
    \nonumber \\
    &\leq&
    \enc{\sum_{j=1}^N\abs{\lam_j-\lam_j^{t-1}}}
    \max_j\abs{\sum_{i=1}^mD_t(i)M_{ij}}\nonumber \\
    &=&
    \label{rate:re_ineq:eqn}
    \delta_t\norm{\vlam-\vlam^{t-1}}_1 .
  \end{eqnarray}
  Since the previous holds for any $\vlam$ suffering less than the
  target loss, the last expression is at most
  $\delta_tS_{t-1}$. Combining this with \eqref{rate:re_ineq:eqn}
  completes the proof.
\end{proof}
To complete the proof of Theorem~\ref{rate:B_rate:thm}, we show $S_t$
is small compared to $R_t$ in rounds $t\leq T_0$ (during which we have
assumed $S_t,R_t$ are all positive). In fact we prove:
\begin{lemma}
  \label{rate:SR:lem}
  For any $t \leq T_0$, $S_t \leq B^3R_t^{-2}$.
\end{lemma}
This, along with Lemmas~\ref{rate:edge_B_rate:lem} and
\ref{rate:B_edge:lem}, immediately proves
Theorem~\ref{rate:B_rate:thm}.  The bound on $S_t$ in
Lemma~\ref{rate:SR:lem} can be proven if we can first show $S_t$ grows
slowly compared to the rate at which the suboptimality $R_t$
falls. Intuitively this holds since growth in $S_t$ is caused by a
large step, which in turn will drive down the suboptimality. In fact
we can prove the following.
\begin{lemma}
  \label{rate:dsdr:lem}
  In any round $t\leq T_0$, we have $\frac{2\dr_t}{R_{t-1}} \geq
  \frac{\ds_t}{S_{t-1}}$.
\end{lemma}
\begin{proof}
  Firstly, it follows from the definition of $S_t$ that $\ds_t \leq
  \norm{\vlam^t - \vlam^{t-1}}_1 = \abs{\alpha_t}$. Next, using
  \eqref{rate:dr:eqn} and \eqref{rate:step:eqn} we may write $\dr_t
  \geq \Upsilon(\delta_t)\abs{\alpha_t}$, where the function
  $\Upsilon$ has been defined in \citep{RatschWa05} as
  \[
  \Upsilon(x) =
  \frac{-\ln(1-x^2)}
  {\ln\enc{\frac{1+x}{1-x}}}.
  \]
  It is known~\citep{RatschWa05,RudinScDa07} that $\Upsilon(x) \geq
  x/2$ for $x\in[0,1]$. Combining and using
  Lemma~\ref{rate:B_edge:lem},
  \[
  \dr_t \geq \delta_t\ds_t/2 \geq R_{t-1}\enc{\ds_t/2S_{t-1}}.
  \]
  Rearranging completes the proof.
\end{proof}
Using this we may prove Lemma~\ref{rate:SR:lem}.

\begin{proof}
  We first show $S_0 \leq B^3R_0^{-2}$.  
Note, $S_0 \leq  \norm{\vlams-\vlam^0}_1 = B$, and by definition the
quantity $R_0= -\ln\enc{\frac{1}{m}\sum_ie^{-(\M\vlams)_i}}$.   
The quantity
  $(\M\vlams)_i$ is the inner product of row $i$ of matrix $\M$ with the
  vector $\vlams$. Since the entries of $\M$ lie in $[-1,+1]$, this is
  at most $\nlam = B$. Therefore $R_0 \leq
  -\ln\enc{\frac{1}{m}\sum_ie^{-B}} = B$, which is what we needed.

  To complete the proof, we show that $R_t^2S_t$ is non-increasing. It
  suffices to show for any $t$ the inequality $R_t^2S_t \leq
  R_{t-1}^2S_{t-1}$. This holds by the following chain:
  \begin{eqnarray*}
    R_t^2S_t &=& \enc{R_{t-1}-\dr_t}^2\enc{S_{t-1}+\ds_t} 
    = R_{t-1}^2S_{t-1}\enc{1-\frac{\dr_t}{R_{t-1}}}^2
    \enc{1+\frac{\ds_t}{S_{t-1}}} \\
    &\leq& R_{t-1}^2S_{t-1}\exp
    \enc{-\frac{2\dr_t}{R_{t-1}} + \frac{\ds_t}{S_{t-1}}}
    \leq R_{t-1}^2S_{t-1},
  \end{eqnarray*}
  where the first inequality follows from $e^x \geq 1+x$, and the
  second one from Lemma~\ref{rate:dsdr:lem}.
\end{proof}
}
\longv{
This completes the proof of Theorem~\ref{rate:B_rate:thm}. Although our
bound provides a rate polynomial in $B,\eps^{-1}$ as desired by the
conjecture in \citep{Schapire10}, the exponents are rather large, and
(we believe) not tight. One possible source of slack is the bound on
$S_t$ in Lemma~\ref{rate:SR:lem}. Qualitatively, the distance $S_t$ to
some solution having target loss should decrease with rounds, whereas
Lemma~\ref{rate:SR:lem} only says it does not increase too
fast. Improving this will directly lead to a faster convergence
rate. In particular, showing that $S_t$ never decreases would imply a
$B^2/\eps$ rate of convergence. Whether or not the monotonicity of
$S_t$ holds, we believe that the obtained rate bound is probably true,
and state it as a conjecture.
\begin{conjecture}
  \label{rate:B_rate:conj}
  For any $\vlams$ and $\eps>0$, AdaBoost converges to within
  $\lf(\vlams)+\eps$ loss in $O(B^2/\eps)$ rounds, where the order
  notation hides only absolute constants.
\end{conjecture}
As evidence supporting the conjecture, we show in the next section how
a minor modification to AdaBoost can achieve the above rate.

\subsection{Faster rates for a variant}
\label{rate:abs:sec}

In this section we introduce a new algorithm, \Abs, which will
enjoy the much faster rate of convergence mentioned in
Conjecture~\ref{rate:B_rate:conj}. \Abs is the same as AdaBoost,
except that at the end of each round, the current combination of weak
hypotheses is {\em scaled back}, that is, multiplied by a scalar in
$[0,1]$ if doing so will reduce the exponential loss further.  The
code is largely the same as in Section~\ref{rate:sec:adab},
maintaining a combination $\vlam^{t-1}$ of weak hypotheses, and
greedily choosing $\alpha_t$ and $\bh_{j_t}$ on each round to form a
new combination $\tilde{\vlam}^t=\vlam^{t-1}+\alpha_t \bh_{j_t}$.
%  In
% fact, while in the previous section we required the step size
% $\alpha_t$ to be chosen according to \eqref{rate:step:eqn}, here we
% will allow any step size as long as the drop in the loss is at least
% $\lf(\tilde{\vlam}^t) \leq \lf(\vlam^{t-1})\sqrt{1-\delta^2}$, where
% $\delta_t$ is the largest possible edge achievable in that
% round.
 However, after creating the new combination $\tilde{\vlam}^t$,
the result is multiplied by the value $s_t$ in $[0,1]$ that causes the
greatest decrease in the exponential loss:
$s_t=\argmin_s\lf(s\tilde{\vlam}^t)$,
and $\vlam^t =
s_t\tilde{\vlam}^t$. Since
$\lf(s\tilde{\vlam}^t)$, as a function of $s$, is convex, its minimum
on $[0,1]$ can be found easily, for instance, using a simple binary
search.  The new distribution $D_{t+1}$ on the examples is constructed
using $\vlam^t$ as before; the weight $D_{t+1}(i)$ on example $i$ is
proportional to its exponential loss $D_{t+1}(i) \propto
e^{-(\M\vlam^t)_i}$. With this modification we may prove the following:
\begin{theorem}
  \label{rate:abs_rate:thm}
  For any $\vlams, \eps>0$, \Abs achieves at most
  $\lf(\vlams)+\eps$ loss within $3\norm{\vlams}_1^2/\eps$ rounds.
\end{theorem}
The proof is similar to that in the previous section. Reusing the same
notation, note that proof of Lemma~\ref{rate:edge_B_rate:lem} continues to hold
(with very minor modifications to that are straightforward). Next we can
exploit the changes in \Abs to show an improved version of
Lemma~\ref{rate:B_edge:lem}. Intuitively, scaling back has the effect
of preventing the weights on the weak hypotheses from becoming ``too
large'', and we may show
\begin{lemma}
  \label{rate:abs_edge:lem}
  In each round $t$, the edge satisfies $\delta_t \geq R_{t-1}/B$. 
\end{lemma}
\begin{proof}
  We will reuse parts of the proof of
  Lemma~\ref{rate:B_edge:lem}. Setting $\vlam = \vlams$ in
  \eqref{rate:re:eqn} we may write
  \[
  R_t \leq \sum_{i=1}^mD_{\vlam^{t-1}}(i)\enc{\M\vlams}_i
  +\sum_{i=1}^m-D_{\vlam^{t-1}}(i)\enc{\M\vlam^{t-1}}_i.
  \]
  The first summation can be upper bounded as in
  \eqref{rate:re_ineq:eqn} by $\delta_t\norm{\vlams} = \delta_t B$. We
  will next show that the second summation is non-positive, which will
  complete the proof. The scaling step was added just so that this
  last fact would be true.

  If we define $G:[0,1]\to\R$ to be $G(s) =
  \lf\enc{s\tilde{\vlam^{t}}} = \sum_i e^{-(\M\tilde{\vlam^t})_8}$,
  then observe that the scaled derivative $G'(s)/G(s)$ is exactly
  equal to the second summation. Since $G(s) \geq 0$, it suffices to
  show the derivative $G'(s) \leq 0$ at the optimum value of $s$,
  denoted by $s^*$. Since $G$ is a strictly convex function ($\forall
  s: G''(s) > 0$), it is either strictly increasing or strictly
  decreasing throughout $[0,1]$, or it has a local minima. In the case
  when it
  is strictly decreasing throughout, then $G'(s)\leq 0$ everywhere,
  whereas if $G$ has a local minima, then $G'(s) = 0$ at $s^*$. We finish
  the proof by showing that $G$ cannot be strictly increasing
  througout $[0,1]$. If it were, we would have $\lf(\tilde{\vlam}^t) =
  G(1) > G(0) = 1$, an impossibility since the loss decreases through
  rounds.
\end{proof}
Lemmas~\ref{rate:edge_B_rate:lem} and \ref{rate:abs_edge:lem} together
now imply Theorem~\ref{rate:abs_rate:thm}, where we used that $2\ln 2
< 3$.

In experiments we ran, the scaling back never occurs. For such
datasets, AdaBoost and \Abs are identical. We believe that even for
contrived examples, the rescaling could happen only a few times,
implying that both AdaBoost and \Abs would enjoy the convergence rates
of Theorem~\ref{rate:abs_rate:thm}. In the next section, we construct
rate lower bound examples to show that this is nearly the best rate one can
hope to show.
}
\old{  
\subsection {Lower-bounds}
\label{rate:B_lbnd:sec}
Here we show that the dependence of the rate in
Theorem~\ref{rate:B_rate:thm} on the norm $\nlam$ of the solution
achieving target accuracy is necessary for a wide class of
datasets. The arguments in this section are not tailored to AdaBoost,
but hold more generally for any coordinate descent algorithm, and can
be readily generalized to any loss function $L'$ of the form
$L'(\vlam) = (1/m)\sum_{i}\phi(\M\vlam)$, where $\phi:\R\to\R$ is any
non-decreasing function. The first lemma connects the size of a
reference solution to the required number of rounds of boosting, and
shows that for a wide variety of datasets the convergence rate to a
target loss can be lower bounded by the $\ell_1$-norm of the smallest
solution achieving that loss.
\begin{lemma}
  \label{rate:B_lbnd:lem}
  Suppose the feature matrix $\M$ corresponding to a dataset has two
  rows with $\set{-1,+1}$ entries which are complements of each other,
  i.e., there are two examples on which any hypothesis gets one wrong
  and one correct prediction.  Then the number of rounds required to
  achieve a target loss $L^*$ is at least $\inf \set{\norm{\vlam}_1:
    \lf(\vlam) \leq L^*}/(2\ln m)$.
\end{lemma}
\begin{proof}
  We first show that the two examples corresponding to the
  complementary rows in $\M$ both satisfy a certain margin boundedness
  property. Since each hypothesis predicts oppositely on these, in any
  round $t$ their margins will be of equal magnitude and opposite
  sign. Unless both margins lie in $[-\ln m, \ln m]$, one of them will
  be smaller than $-\ln m$. But then the exponential loss
  $\lf(\vlam^t) = (1/m)\sum_j e^{-(\M\vlam^t)_j}$ in that round will
  exceed $1$, a contradiction since the losses are non-increasing
  through rounds, and the loss at the start was $1$. Thus, assigning
  one of these examples the index $i$, we have the absolute margin
  $\abs{(\M\vlam^{t})_i}$ is bounded by $\ln m$ in any round
  $t$. Letting $\M(i)$ denote the $i$th row of $\M$, the step length
  $\alpha_t$ in round $t$ therefore satisfies
  \[
  \abs{\alpha_t} = \abs{M_{ij_t}\alpha_t}
  = \abs{\dotp{\M(i)}{\alpha_t\ve_{j_t}}}
  = \abs{(\M\vlam^t)_i - (\M\vlam^{t-1})_i}
  \leq \abs{(\M\vlam^t)_i} + \abs{(\M\vlam^{t-1})_i} \leq 2\ln m,
  \]
  and the statement of the lemma directly follows.
\end{proof}
}
\longv{
When the weak hypotheses are \emph{abstaining}~\citep{SchapireSi99},
it can make a definitive prediction that the label is $-1$ or $+1$, or
it can ``abstain'' by predicting zero. No other levels of confidence
are allowed, and the resulting feature matrix has entries in
$\set{-1,0,+1}$. The next theorem constructs a feature matrix
satisfying the properties of Lemma~\ref{rate:B_lbnd:lem} and where
additionally the smallest size of a solution achieving $L^*+\eps$ loss
is at least $\Omega(2^m)\ln(1/\eps)$, for some fixed $L^*$ and every
$\eps > 0$.
\begin{theorem}
  \label{rate:B_big:thm} % big B
  Consider the following matrix $\M$ with $m$ rows (or examples)
  labeled $0,\ldots,m-1$ and $m-1$ columns labeled $1,\ldots,m-1$
  (assume $m\geq 3$). The square sub-matrix ignoring row zero is an
  upper triangular matrix, with $1$'s on the diagonal, $-1$'s above
  the diagonal, and $0$ below the diagonal. Therefore row 1 is
  $(+1,-1,-1,\ldots,-1)$. Row 0
  is defined to be just the complement of row 1. Then, for any
  $\eps>0$, a loss of $2/m + \eps$ is achievable on this dataset, but
  with large norms
  \[
  \inf\set{\norm{\vlam}_1: \lf(\vlam) \leq 2/m  + \eps} \geq
  (2^{m-2}-1)\ln(1/(3\eps)).
  \]
  Therefore, by Lemma~\ref{rate:B_lbnd:lem}, the minimum number of
  rounds required for reaching loss at most $2/m+\eps$ is at least
  $\enc{\frac{2^{m-2}-1}{2\ln m}}\ln(1/(3\eps))$.
\end{theorem}
A picture of the matrix constructed in the above lemma for $m=5$ is
shown in Figure~\ref{rate:matrix:fig}.
\begin{figure}
  \centering
  \[
  \left (
  \begin{array}{ccccc}
    - & + & + & + & + \\
    + & - & - & - & - \\
    0 & + & - & - & - \\
    0 & 0 & + & - & - \\
    0 & 0 & 0 & + & - \\
    0 & 0 & 0 & 0 & +
  \end{array}
  \right )
  \]
  \caption{The matrix used in Theorem~\ref{rate:B_big:thm} when $m=5$.}
  \label{rate:matrix:fig}
\end{figure}
Theorem~\ref{rate:B_big:thm} shows that when $\eps$ is a small constant
(say $\eps = 0.01$), and $\vlams$ is some vector with loss $L^* +
\eps/2$, AdaBoost takes 
at least $\Omega(2^m/\ln m)$ steps to get within $\eps/2$ of the loss
achieved by $\vlams$, that is, to within $L^* + \eps$ loss.
Since $m$ and $\eps$ are independent quantities, this shows that a
polynomial dependence on the norm of 
the reference solution is unavoidable, and this norm might be
exponential in the number of training examples in the worst case.
\begin{corollary}
  \label{rate:B_big:cor}
  Consider feature matrices containing only $\set{-1,0,+1}$ entries.
  If, for some constants $c$ and $\beta$, the bound in
  Theorem~\ref{rate:B_rate:thm} can be replaced by
  $O\enc{\nlam^c\eps^{-\beta}}$ for all such matrices, then $c\geq
  1$. Further, for such matrices, the bound $\poly(1/\eps, \nlam)$ in
  Theorem~\ref{rate:B_rate:thm} cannot be replaced by
  $\poly(1/\eps,m,N)$.
\end{corollary}
We now prove Theorem~\ref{rate:B_big:thm}.
\old{
\begin{proof}{\bf of Lemma~\ref{rate:B_big:thm}}.
  We first lower bound the norm of solutions achieving loss at most
  $2/m + \eps$. Observe that since rows 0 and 1 are complementary,
  any solution's loss on just examples 0 and 1 will add up to at least
  $2/m$. Therefore, to get within $2/m+\eps$, the margins on
  examples $2,\ldots,m-1$ should be at least $\ln
  \enc{\enc{m-2}/\enc{m\eps}} \geq \ln(1/(3\eps))$ (for $m\geq
  3$).
  Now, the feature matrix is designed so that the margins due to a
  combination $\vlam$ satisfy the following recursive relationships:
  \begin{eqnarray*}
    (M\vlam)_{m-1} &=& \lam_{m-1}, \\
    (M\vlam)_{i} &=& \lam_{i} - \enc{\lam_{i+1} + \ldots +
      \lam_{m-1}}, \mbox { for } 1 \leq i \leq m-2.
  \end{eqnarray*}
%   Now, if a solution $\vlam$ gets margin at least
%   $\ln(1/(3\eps))$ on example $m-1$, then $\lam_{m-1} \geq \ln(1/(3\eps))$
%   (since the other columns get zero margin on it). Since column $m-1$
%   gets margin $-1$ on example $m-2$, and column $m-2$ is the only
%   column with a positive margin on that example, the previous fact
%   forces $\lam_{m-2} \geq \ln(1/(3\eps)) + \lam_{m-1} \geq
%   2\ln(1/(3\eps))$.
  Therefore, the margin on example $m-1$ is at least $\ln(1/(3\eps))$
  implies $\lam_{m-1} \geq \ln(1/(3\eps))$. 
  Similarly,  $\lam_{m-2} \geq  \ln(1/(3\eps)) + \lam_{m-1} \geq 
  2\ln(1/(3\eps))$. 
  Continuing this way, 
  \[
  \lam_i \geq
  \ln\enc{\frac{1}{3\eps}} + \lam_{i+1} + \ldots + \lam_{m-1} \geq
  \ln\enc{\frac{1}{3\eps}}\enct{1 + 2^{(m-1)-(i+1)} + \ldots + 2^0} =
  \ln\enc{\frac{1}{3\eps}}2^{m-1-i},
  \] 
  for $i=m-1,\ldots,2$. Hence $ \norm{\vlam}_1 \geq
  \ln(1/(3\eps))(1+2+\ldots+2^{m-3}) = (2^{m-2}-1)\ln(1/(3\eps))$. 

  We end by showing that a loss of at most $2/m+\eps$ is
  achievable.
  The above argument implies that if $\lam_{i} =
  2^{m-1-i}$ for $i=2,\ldots, m-1$, then examples $2,\ldots,m-1$
  attain margin exactly $1$. If we choose $\lam_1 = \lam_2 + \ldots +
  \lam_{m-1} = 2^{m-3} + \ldots + 1 = 2^{m-2}-1$, then the recursive
  relationship implies a zero margin on example 1 (and hence example 0).
  Therefore the combination $\ln(1/\eps)(2^{m-2}-1,2^{m-3},2^{m-4},\ldots,1)$
  achieves a loss $(2+(m-2)\eps)/m \leq 2/m + \eps$, for any $\eps>0$.
\end{proof}
}
We finally show that if the weak hypotheses are confidence-rated with
arbitrary levels of confidence, so that the feature matrix is allowed
to have non-integral entries in $[-1,+1]$, then the minimum norm of a
solution achieving a fixed accuracy can be arbitrarily large. Our
constructions will satisfy the requirements of
Lemma~\ref{rate:B_lbnd:lem}, so that the norm lower bound translates
into a rate lower bound.
\begin{figure}
  \centering
  \[
  \left (
    \begin{array}{ll}
      -1 & +1 \\
      +1 & -1 \\
      -1+\nu & +1 \\
      +1 & -1+\nu
      \end{array}
  \right )
  \]
  \caption{A picture of the matrix used in
    Theorem~\ref{rate:B_nonint_big:thm}.}
  \label{rate:matrix_nonint:fig}
\end{figure}
\begin{theorem}
  \label{rate:B_nonint_big:thm}
  Let $\nu > 0$ be an arbitrary number, and let $\M$ be the (possibly)
  non-integral matrix with 4 examples and 2 weak hypotheses shown in
  Figure~\ref{rate:matrix_nonint:fig}. Then for any $\eps>0$, a loss
  of $1/2+\eps$ is achievable on this dataset, but with large norms
  \[
  \inf\set{\norm{\vlam}_1: \lf(\vlam) \leq 1/2 + \eps} \geq
  2\ln(1/(2\eps))\nu^{-1}.
  \]
  Therefore, by Lemma~\ref{rate:B_lbnd:lem}, the number of rounds
  required to achieve loss at most $1/2 + \eps$ is at least
  $\ln(1/(2\eps))\nu^{-1}/\ln(m)$.
\end{theorem}
\begin{proof}
  We first show a loss of $1/2+\eps$ is achievable. Observe that the
  vector $\vlam = (c,c)$, with $c=\nu^{-1}\ln(1/(2\eps))$, achieves
  margins $0,0,\ln(1/(2\eps)),\ln(1/(2\eps))$ on examples $1,2,3,4$,
  respectively. Therefore $\vlam$ achieves loss $1/2+\eps$. We next
  show a lower bound on the norm of a solution achieving this
  loss. Observe that since the first two rows are complementary, the
  loss due to just the first two examples is at least
  $1/2$. Therefore, any solution $\vlam = (\lam_1,\lam_2)$ achieving
  at most $1/2+\eps$ loss overall must achieve a margin of at least
  $\ln(1/(2\eps))$ on both the third and fourth examples. By
  inspecting the two columns, this implies
  \begin{eqnarray*}
    \lam_1 - \lam_2 + \lam_2\nu &\geq& \ln\enc{1/(2\eps)} \\
    \lam_2 - \lam_1 + \lam_1\nu &\geq& \ln\enc{1/(2\eps)}.
  \end{eqnarray*}
  Adding the two equations we find
  \[
  \nu(\lam_1 + \lam_2) \geq 2\ln\enc{1/(2\eps)}
  \implies
  \lam_1 + \lam_2 \geq 2\nu^{-1}\ln\enc{1/(2\eps)}.
  \]
  By the triangle inequality, $\norm{\vlam}_1 \geq \lam_1 + \lam_2$,
  and the lemma follows.
\end{proof}
Note that if $\nu=0$, then the optimal solution is found in zero
rounds of boosting and has optimal loss $1$. However, even the tiniest
perturbation $\nu > 0$ causes the optimal loss to fall to $1/2$, and
causes the rate of convergence to increase drastically. In fact, by
Theorem~\ref{rate:B_nonint_big:thm}, the number of rounds required to
achieve any fixed loss below $1$ grows as $\Omega(1/\nu)$, which is
arbitrarily large when $\nu$ is infinitesimal. We may conclude that
with non-integral feature matrices, the dependence of the rate on the
norm of a reference solution is absolutely necessary.
\begin{corollary}
  \label{rate:B_nonint_big:cor} When using confidence rated
  weak-hypotheses with arbitrary confidence levels, the bound
  $\poly(1/\eps,\nlam)$ in Theorem~\ref{rate:B_rate:thm} cannot be
  replaced by any function of purely $m$, $N$ and $\eps$ alone.
\end{corollary}
The construction in Figure~\ref{rate:matrix_nonint:fig} can be
generalized to produce datasets with any number of examples that
suffer the same poor rate of convergence as the one in
Theorem~\ref{rate:B_nonint_big:thm}. We discussed the smallest such
construction, since we feel that it best highlights the drastic effect
non-integrality can have on the rate.

In this section we saw how the norm of the reference solution is an
important parameter for bounding the convergence rate. In the next
section we investigate the optimal dependence of the rate on the
parameter $\eps$ and show that $\Omega(1/\eps)$ rounds are necessary
in the worst case.
}
\longv{
\section{Second convergence rate: Convergence to optimal loss}
\label{rate:eps:sec}

In the previous section, our rate bound depended on both the
approximation parameter $\eps$, as well as the size of the smallest
solution achieving the target loss. For many datasets, the optimal
target loss $\inf_{\vlam}\lf(\vlam)$ cannot be realized by any finite
solution. In such cases, if we want to bound the number of rounds
needed to achieve within $\eps$ of the optimal loss, the only way to
use Theorem~\ref{rate:B_rate:thm} is to first decompose the accuracy
parameter $\eps$ into two parts $\eps = \eps_1 + \eps_2$, find some
finite solution $\vlams$ achieving within $\eps_1$ of the optimal
loss, and then use the bound $\poly(1/\eps_2,\nlam)$ to achieve at
most $\lf(\vlams) + \eps_2 = \inf_{\vlam}\lf(\vlam)+\eps$
loss. However, this introduces implicit dependence on $\eps$ through
$\nlam$ which may not be immediately clear. In this section, we show
bounds of the form $C/\eps$, where the constant $C$ depends only on
the feature matrix $\M$, and not on $\eps$. Additionally, we show that
this dependence on $\eps$ is optimal in Lemma~\ref{rate:eps_lbnd:lem}
of the Appendix, where $\Omega(1/\eps)$ rounds are shown to be
necessary for converging to within $\eps$ of the optimal loss on a
certain dataset. Finally, we note that the lower bounds in the
previous section indicate that $C$ can be $\Omega(2^m)$ in the worst
case for integer matrices (although it will typically be much
smaller), and hence this bound, though stronger than that of
Theorem~\ref{rate:B_rate:thm} with respect to $\eps$, cannot be used
to prove the conjecture in \citep{Schapire10}, since the constant is
not polynomial in the number of examples $m$.
\subsection{Upper Bound}
\label{rate:ub2:sec}
The main result of this section is the following rate upper
bound. A similar approach to
solving this problem was taken independently by \citet{Telgarsky11}.
\begin{theorem}
  \label{rate:rate:thm}
  AdaBoost reaches within $\eps$ of the optimal loss in at most
  $C/\eps$ rounds, where $C$ only depends on the feature matrix.
\end{theorem}
}
\old{
Our techniques build upon earlier work on the rate of convergence of
AdaBoost, which have mainly considered two particular cases. In the
first case, the \emph{weak learning assumption} holds, that is, the
edge in each round is at least some fixed constant. In this situation,
\citet{FreundSc97} and \citet{SchapireSi99} show that the optimal loss
is zero, that no solution with finite size can achieve this loss, but
AdaBoost achieves at most $\eps$ loss within $O(\ln(1/\eps))$
rounds. In the second case some finite combination of the weak
classifiers achieves the optimal loss, and \citet{RatschMiWa02}, using
results from \citet{LuoTs92}, show that AdaBoost achieves within
$\eps$ of the optimal loss again within $O(\ln(1/\eps))$ rounds.

Here we consider the most general situation, where the weak learning
assumption may fail to hold, and yet no finite solution may achieve
the optimal loss. The dataset used in Lemma~\ref{rate:eps_lbnd:lem}
and shown in Figure~\ref{rate:dat:fig} exemplifies this situation. Our
main technical contribution shows that the examples in any dataset can
be partitioned into a \emph{zero-loss set} and \emph{finite-margin
  set}, such that a certain form of the weak learning assumption holds
within the zero-loss set, while the optimal loss considering only the
finite-margin set can be obtained by some finite solution. The two
partitions provide different ways of making progress in every round,
and one of the two kinds of progress will always be sufficient for us
to prove Theorem~\ref{rate:rate:thm}.

We next state our decomposition result, illustrate it with an example,
and then state several lemmas quantifying the nature of the progress
we can make in each round. Using these lemmas, we prove
Theorem~\ref{rate:rate:thm}.
\begin{lemma}(Decomposition Lemma)
  \label{rate:dec:lem}
  For any dataset, there exists a partition of the set of training
  examples $X$ into a (possibly empty) \emph{zero-loss set} $\zl$ and a
  (possibly empty) \emph{finite-margin set} $\fl=\zl^c\eqdef X\setminus \zl$
  such that the following hold simultaneously :
  \begin{enumerate}
  \item \label{rate:one:dec} For some positive constant $\gamma > 0$,
    there exists some vector $\alf$ with unit $\ell_1$-norm
    $\norm{\alf}_1 = 1$ that attains at least $\gamma$ margin on each
    example in $\zl$, and exactly zero margin on each example in $\fl$
    \begin{eqnarray*}
      \forall i\in \zl: (\M\alf)_i \geq \gamma, &&
      \forall i\in \fl: (\M\alf)_i = 0.
    \end{eqnarray*}
  \item \label{rate:two:dec} The optimal loss considering only
    examples within $\fl$ is achieved by some finite combination
    $\vals$.
  \item \label{rate:three:dec} There is a constant $\mu_{\max} <
    \infty$, such that for any combination $\val$ with bounded loss on
    the finite-margin set, $\sum_{i\in \fl}e^{-(\M\val)_i} \leq m$, the
    margin $(\M\val)_i$ for any example $i$ in $\fl$ lies in the bounded
    interval $[-\ln m, \mu_{\max}]$.
  \end{enumerate}
\end{lemma}
A proof is deferred to the next section. The decomposition~lemma
immediately implies that the vector $\vals + \infty\cdot\alf$, which
denotes $\enc{\vals + c\alf}$ in the limit $c\to\infty$, is an optimal
solution, achieving zero loss on the zero-loss set, but only finite
margins (and hence positive losses) on the finite-margin set (thereby
justifying the names).

\begin{wrapfigure}{l}{0.3\textwidth}%[ht]
  \centering
  \begin{tabular}{c|c|c}
    & $\bh_1$ & $\bh_2$ \\
    \hline 
    $a$ & $+$ & $-$ \\
    $b$ & $-$ & $+$ \\
    $c$ & $+$ & $+$ 
  \end{tabular}
  \caption{A dataset requiring $\Omega(1/\eps)$ rounds for convergence.}
  \label{rate:dat:fig}
  %\vspace{-0.5em}
\end{wrapfigure}
Before proceeding, we give an example dataset and indicate the
zero-loss set, finite-margin set, $\vals$ and $\alf$ to illustrate our
definitions. Consider a dataset with three examples $\set{a,b,c}$ and
two hypotheses $\set{\bh_1,\bh_2}$ and the feature matrix $\M$ in
Figure~\ref{rate:dat:fig}. Here $+$ means correct ($M_{ij}=+1$) and
$-$ means wrong ($M_{ij}=-1$). The optimal solution is $\infty\cdot
(\bh_1 + \bh_2)$ with a loss of $2/3$. The finite-margin set is
$\set{a,b}$, the zero-loss set is $\set{c}$, $\alf = (1/2,1/2)$ and
$\vals=(0,0)$; for this dataset these are unique. This dataset also
serves as a lower-bound example in Lemma~\ref{rate:eps_lbnd:lem},
where we show that $2/(9\eps)$ rounds are necessary for AdaBoost to
achieve loss at most $(2/3)+\eps$.

Before providing proofs, we introduce some notation.  By
$\norm{\cdot}$ we will mean $\ell_2$-norm; every other norm will have
an appropriate subscript, such as $\norm{\cdot}_1,
\norm{\cdot}_\infty$, etc. The set of all training examples will be
denoted by $X$. By $\loss{\vlam}{i}$ we mean the exp-loss
$e^{-(\M\vlam)_i}$ on example $i$. For any subset $S\subseteq X$ of
examples, $\loss{\vlam}{S} = \sum_{i\in S}\loss{\vlam}{i}$ denotes the
total exp-loss on the set $S$. Notice $\lf(\vlam) =
(1/m)\loss{\vlam}{X}$, and that $ D_{t+1}(i) =
\loss{\vlam^{t}}{i}/\loss{\vlam^{t}}{X}$, where $\vlam^t$ is the
combination found by AdaBoost at the end of round $t$. By
$\leds{\val}{\vlam}{S}$ we mean the edge obtained on the set $S$ by
the vector $\val$, when the weights over the examples are given by
$\loss{\vlam}{\cdot} / \loss{\vlam}{S}$:
\[
\leds{\val}{\vlam}{S} = \abs{\frac{1}{\loss{\vlam}{S}}\sum_{i\in
    S}\loss{\vlam}{i}(\M\val)_i}.
\]
In the rest of the section, by ``loss'' we mean the unnormalized loss
$\loss{\vlam}{X} = m\lf(\vlam)$ and show that in $C/\eps$ rounds AdaBoost
converges to within $\eps$
of the optimal unnormalized loss $\inf_{\vlam}\loss{\vlam}{X}$,
henceforth denoted by $K$. 
Note that this means AdaBoost takes $C/\eps$ rounds to converge to
within $\eps/m$ of the optimal normalized loss, that is to loss at most
$\inf_{\vlam}\lf(\vlam) + \eps/m$.
Replacing $\eps$ by $m\eps$, it takes $C/(m\eps)$ steps to attain
normalized loss at most $\inf_{\vlam}\lf(\vlam) + \eps$. Thus, whether
we use normalized or unnormalized does not substantively affect the
result in Theorem~\ref{rate:rate:thm}. 
The progress due to the zero-loss set is
now immediate from Item~\ref{rate:one:dec} of the decomposition lemma:
\begin{lemma}
  \label{rate:edgestep:lem}
  In any round $t$, the maximum edge $\delta_t$ is at least
  $\gamma\loss{\vlam^{t-1}}{\zl}/\loss{\vlam^{t-1}}{X}$,
  where $\gamma$ is as in Item~\ref{rate:one:dec} of the decomposition
  lemma.
\end{lemma}
\begin{proof}
  Recall the distribution $D_t$ created by AdaBoost in round $t$ puts
  weight $D_t(i) = \loss{\vlam^{t-1}}{i}/\loss{\vlam^{t-1}}{X}$ on
  each example $i$. From Item~\ref{rate:one:dec} we get
\[
\led{\alf}{\vlam^{t-1}}
=
\abs{\frac{1}{\loss{\vlam^{t-1}}{X}}\sum_{i\in X}\loss{\vlam^{t-1}}{i}(\M\alf)_i}
=
\frac{1}{\loss{\vlam^{t-1}}{X}}\sum_{i\in \zl}\gamma\loss{\vlam^{t-1}}{i}
= \gamma\enc{\frac{\loss{\vlam^{t-1}}{\zl}}{\loss{\vlam^{t-1}}{X}}}.
\]
Since $(\M\alf)_i = \sum_j \eta^{\dagger}_j(\M\ve_j)_i$, we may
rewrite the edge $\led{\alf}{\vlam^{t-1}}$ as follows:
\begin{eqnarray*}
 \led{\alf}{\vlam^{t-1}}
&=&
\abs{\frac{1}{\loss{\vlam^{t-1}}{X}}\sum_{i\in X}\loss{\vlam^{t-1}}{i}
\sum_j \eta^{\dagger}_j(\M\ve_j)_i}\\
&=&
\abs{\sum_j \eta^{\dagger}_j
\frac{1}{\loss{\vlam^{t-1}}{X}}\sum_{i\in X}
\loss{\vlam^{t-1}}{i}(\M\ve_j)_i}\\
&=&
\abs{\sum_j \eta^{\dagger}_j \led{\ve_j}{\vlam^{t-1}}}
\leq \sum_j \abs{\eta^{\dagger}_j}\led{\ve_j}{\vlam^{t-1}}.
\end{eqnarray*}
Since the $\ell_1$-norm of $\alf$ is $1$, the weights
$\abs{\eta^{\dagger}_j}$ form some distribution $p$ over the columns
$1,\ldots,N$. We may therefore conclude
 \[
 \gamma\enc{\frac{\loss{\vlam^{t-1}}{\zl}}{\loss{\vlam^{t-1}}{X}}}
 =
 \led{\alf}{\vlam^{t-1}}
 \leq
 \E_{j\sim p}\enco{\led{\ve_j}{\vlam^{t-1}}}
 \leq
 \max_{j}\led{\ve_j}{\vlam^{t-1}} \leq \delta_t.
 \]
\end{proof}
If the set $\fl$ were empty, then Lemma~\ref{rate:edgestep:lem} implies
an edge of $\gamma$ is available in each round. This in fact means
that the weak learning assumption holds, and using
\eqref{rate:drop:eqn}, we can show an $O(\ln(1/\eps)\gamma^{-2})$ bound
matching the rate bounds of \citet{FreundSc97} and
\citet{SchapireSi99}. So henceforth, we assume that $\fl$ is
non-empty. Note that this implies that the optimal loss $K$ is at
least $1$ (since any solution will get non-positive margin on some
example in $\fl$), a fact we will use later in the proofs.

Lemma~\ref{rate:edgestep:lem} says that the edge is large if the loss
on the zero-loss set is large. On the other hand, when it is small,
Lemmas~\ref{rate:taylor:lem} and \ref{rate:finstep:lem} together show
how AdaBoost can make good progress using the finite margin
set. Lemma~\ref{rate:taylor:lem} uses second order methods to show how
progress is made in the case where there is a finite solution. Similar
arguments, under additional assumptions, have earlier appeared in
\citep{RatschMiWa02}.
\begin{lemma}
  \label{rate:taylor:lem}
  Suppose $\vlam$ is a combination such that $m \geq \loss{\vlam}{\fl}
  \geq K$.  Then in some coordinate direction the edge is at least
  $\sqrt{C_0\enc{\loss{\vlam}{\fl}-K}/\loss{\vlam}{\fl}}$, where $C_0$ is
  a constant depending only on the feature matrix $\M$.
\end{lemma}
\begin{proof}
  Let $\M_\fl\in \R^{|\fl|\times N}$ be the matrix $\M$ restricted to only
  the rows corresponding to the examples in $\fl$. Choose $\val$ such
  that $\vlam + \val = \vals$ is an optimal solution over $\fl$. Without
  loss of generality assume that $\val$ lies in the orthogonal
  subspace of the null-space $\set{\vu: \M_\fl\vu = \vzero}$ of $\M_\fl$
  (since we can translate $\vals$ along the null space if necessary
  for this to hold). If $\val = \vzero$, then $\loss{\vlam}{\fl} = K$
  and we are done. Otherwise $\norm{\M_\fl\val} \geq \lmin\norm{\val}$,
  where $\lmin^2$ is the smallest positive eigenvalue of the symmetric
  matrix $\M_\fl^T\M_\fl$ (exists since $\M_\fl\val \neq \vzero$). Now define
  $f:[0,1] \to \R$ as the loss along the (rescaled) segment
  $[\vals,\vlam]$
\[
f(x) \eqdef \loss{(\vals- x\val)}{\fl}
= \sum_{i\in \fl} \loss{\vals}{i}e^{x(\M\val)_i}.
\]
This implies that $f(0) = K$ and $f(1) = \loss{\vlam}{\fl}$. Notice that
the first and second derivatives of $f(x)$ are given by:
\begin{eqnarray*}
  f'(x) = \sum_{i\in \fl} (\M_\fl\val)_i\loss{(\vals - x\val)}{i},
  &&
  f''(x) = \sum_{i\in \fl} (\M_\fl\val)^2_i\loss{(\vals - x\val)}{i}.
\end{eqnarray*}
We next lower bound possible values of the second derivative as follows:
\[
f''(x) =  \sum_{i'\in \fl} (\M_\fl\val)^2_{i'}\loss{(\vals - x\val)}{i'}
\geq \sum_{i'\in \fl} (\M_\fl\val)^2_{i'}\min_{i}\loss{(\vals - x\val)}{i}
\geq \norm{\M_\fl\val}^2\min_{i}\loss{(\vals-x\val)}{i}.
\]
Since both $\vlam = \vals - \val$, and $\vals$ suffer total loss at
most $m$, by convexity, so does $\vals - x\val$ for any
$x\in[0,1]$. Hence we may apply Item~\ref{rate:three:dec} of the
decomposition lemma to the vector $\vals-x\val$, for any $x\in[0,1]$,
to conclude that $\loss{(\vals-x\val)}{i} =
\exp\enct{-(\M_\fl(\vals-x\val))_i} \geq e^{-\mu_{\max}}$ on every
example $i$. Therefore we have,
\[
f''(x) \geq \norm{\M_\fl\val}^2e^{-\mu_{\max}} \geq
\lmin^2e^{-\mu_{\max}}\norm{\val}^2
\mbox{ (by choice of $\val$) }.
\]
A standard second-order result is \citep[see
e.g.][eqn.~(9.9)]{BoydVa04}
\[
\abs{f'(1)}^2 \geq 2\enc{\inf_{x\in[0,1]}f''(x)}\enc{f(1) - f(0)}.
\]
Collecting our results so far, we get
\[
\sum_{i\in \fl}\loss{\vlam}{i}(\M\val)_i = \abs{f'(1)}
\geq \norm{\val}
\sqrt{2\lmin^2e^{-\mu_{\max}}\enc{\loss{\vlam}{\fl}-K}}.
\]
Next let $\oval = \val/\norm{\val}_1$ be $\val$ rescaled to have
unit $\ell_1$ norm. Then we have
\[
\sum_{i\in \fl}\loss{\vlam}{i}(\M\oval)_i
= \frac{1}{\norm{\val}_1}\sum_i\loss{\vlam}{i}(\M\val)_i
\geq  \frac{\norm{\val}}{\norm{\val}_1}
\sqrt{2\lmin^2e^{-\mu_{\max}}\enc{\loss{\vlam}{\fl}-K}}.
\]
Applying the Cauchy-Schwarz inequality, we may lower bound
$\frac{\norm{\val}}{\norm{\val}_1}$ by $1/\sqrt{N}$ (since $\val\in
\R^N$). Along with the fact $\loss{\vlam}{\fl}\leq m$, we may write
\[
\frac{1}{\loss{\vlam}{\fl}}\sum_{i\in \fl}\loss{\vlam}{i}(\M\oval)_i
\geq
\sqrt{2\lmin^2N^{-1}m^{-1}e^{-\mu_{\max}}}
\sqrt{\enc{\loss{\vlam}{\fl}-K}/\loss{\vlam}{\fl}}.
\]
If we define $p$ to be a distribution on the columns
$\set{1,\ldots,N}$ of $\M_\fl$ which puts probability $p(j)$ proportional
to $\abs{\oval_j}$ on column $j$, then we have
\[
\frac{1}{\loss{\vlam}{\fl}}\sum_{i\in \fl}\loss{\vlam}{i}(\M\oval)_i 
\leq \E_{j\sim p}
\abs{\frac{1}{\loss{\vlam}{\fl}}
  \sum_{i\in \fl}\loss{\vlam}{i}(\M\ve_j)_i} \leq
\max_{j}\abs{\frac{1}{\loss{\vlam}{\fl}}
\sum_{i\in \fl}\loss{\vlam}{i}(\M\ve_j)_i}.
\]
Notice the quantity inside the max is precisely the edge
$\leds{\ve_j}{\vlam}{\fl}$ in direction $j$. Combining everything, the
maximum possible edge is 
\[
\max_j \leds{\ve_j}{\vlam}{\fl} \geq
\sqrt{C_0\enc{\loss{\vlam}{\fl}-K}/\loss{\vlam}{\fl}},
\]
where we define $C_0 = 2\lmin^2N^{-1}m^{-1}e^{-\mu_{\max}}$.
\end{proof}
\begin{lemma}
  \label{rate:finstep:lem}
  Suppose, at some stage of boosting, the combination found by
  AdaBoost is $\vlam$, and the loss is $K+\theta$. Let $\dth$ denote
  the drop in the suboptimality $\theta$ after one more round; i.e.,
  the loss after one more round is $K+\theta-\dth$. Then there are
  constants $C_1, C_2$ depending only on the feature matrix (and not
  on $\theta$), such that if $\loss{\vlam}{\zl} < C_1\theta$, then $\dth
  \geq C_2\theta$.
\end{lemma}
\begin{proof}
  Let $\vlam$ be the current solution found by boosting.  Using
  Lemma~\ref{rate:taylor:lem}, pick a direction $j$ in which the edge
  $\leds{\ve_j}{\vlam}{\fl}$ restricted to the finite loss set is at
  least $\sqrt{2C_0(\loss{\vlam}{\fl}-K)/\loss{\vlam}{\fl}}$. We can bound
  the edge $\led{\ve_j}{\vlam}$ on the entire set of examples as
  follows:
  \begin{eqnarray*}
    \led{\ve_j}{\vlam} &=&
    \frac{1}{\loss{\vlam}{X}}\abs{\sum_{i\in
        \fl}\loss{\vlam}{i}(\M\ve_j)_i
      + \sum_{i\in
        \zl}\loss{\vlam}{i}(\M\ve_j)_i} \\
    &\geq&\frac{1}{\loss{\vlam}{X}}
    \enc{\abs{\loss{\vlam}{\fl}\leds{\ve_j}{\vlam}{\fl}}
      - \sum_{i\in \zl}\loss{\vlam}{i}}
    \mbox{ (using the triangle inequality)} \\
    &\geq&\frac{1}{\loss{\vlam}{X}}\enc{
      \sqrt{2C_0(\loss{\vlam}{\fl}-K)\loss{\vlam}{\fl}}
      - \loss{\vlam}{\zl}}.
  \end{eqnarray*}
  Now, $\loss{\vlam}{\zl} < C_1\theta$, and $\loss{\vlam}{\fl}-K = \theta
  - \loss{\vlam}{\zl} \geq (1-C_1)\theta$. Further, we will choose $C_1
  < 1$, so that $\loss{\vlam}{\fl} \geq K \geq 1$. Hence, the previous
  inequality implies
  \[
  \led{\ve_j}{\vlam} \geq \frac{1}{K+\theta}\enc{
    \sqrt{2C_0(1-C_1)\theta} - C_1\theta}.
  \]
  Set $C_1 = \min\set{1/2,(1/4)\sqrt{C_0/(2m)}}$. Using $\theta \leq
  K+\theta = \loss{\vlam}{X} \leq m$, we can bound the square of the
  term in brackets on the previous line as
  \begin{eqnarray*}
    \enc{\sqrt{2C_0(1-C_1)\theta} - C_1\theta}^2
    &\geq&
      2C_0(1-C_1)\theta - 2C_1\theta\sqrt{2C_0(1-C_1)\theta}\\
    &\geq&
    2C_0(1-1/2)\theta 
    - 2\enc{(1/4)\sqrt{C_0/(2m)}}\theta\sqrt{2C_0(1-0)m}\\
    &=&
    C_0\theta/2.
  \end{eqnarray*}
  So, if $\delta$ is the maximum edge in any direction, then
  \[
  \delta
  \geq
  \led{\ve_j}{\vlam} \geq \sqrt{C_0\theta/(2(K+\theta)^2)}
  \geq
  \sqrt{C_0\theta/(2m(K+\theta))},
  \]
  where, for the last inequality, we again used $K+\theta \leq m$.
  Therefore the loss after one more step is at most
  $(K+\theta)\sqrt{1-\delta^2} \leq (K+\theta)(1-\delta^2/2) \leq
  K+\theta - \frac{C_0}{4m}\theta$. Setting $C_2 = C_0/(4m)$ completes the
  proof.
 \end{proof}

 \noindent{\bf Proof of Theorem~\ref{rate:rate:thm}.} At any stage of
 boosting, let $\vlam$ be the current combination, and $K+\theta$ be
 the current loss. We show that the new loss is at most $K+\theta -
 \dth$ for $\dth \geq C_3\theta^2$ for some constant $C_3$ depending
 only on the dataset (and not $\theta$). To see this, either $\loss{\vlam}{\zl} < C_1\theta$, in which case
 Lemma~\ref{rate:finstep:lem} applies, and $\dth \geq C_2\theta \geq
 (C_2/m)\theta^2$ (since $\theta = \loss{\vlam}{X}-K \leq m$). Or
 $\loss{\vlam}{\zl} \geq C_1\theta$, in which case applying
 Lemma~\ref{rate:edgestep:lem} yields $\delta \geq \gamma
 C_1\theta/\loss{\vlam}{X} \geq (\gamma C_1/m)\theta$. By
 \eqref{rate:drop:eqn}, $\dth \geq \loss{\vlam}{X}(1-\sqrt{1-\delta^2}) 
 \geq \loss{\vlam}{X}\delta^2/2 \geq (K/2)(\gamma
 C_1/m)^2\theta^2$. Using $K\geq 1$ and choosing $C_3$ appropriately
 gives the required condition. 

 If $K+\theta_t$ denotes the loss in round $t$, then the above claim
 implies $\theta_t - \theta_{t+1} \geq C_3\theta_t^2$. 
 Applying Lemma~\ref{rate:rec:lem} to the sequence $\enct{\theta_t}$ we have
 $1/\theta_T - 1/\theta_0 \geq C_3T$ for any $T$. Since $\theta_0 \geq
 0$, we have $T \leq 1/(C_3\theta_T)$. 
 Hence to achieve loss $K+\eps$, $C_3^{-1}/\eps$ rounds suffice.
  \qed

  \subsection{Proof of the decomposition lemma}
  \label{rate:dec:sec}
 Throughout this section we only consider (unless otherwise stated)
 \emph{admissible} combinations $\vlam$ of weak classifiers, which
 have loss $\loss{\vlam}{X}$ bounded by $m$ (since such are the ones
 found by boosting).  We prove Lemma~\ref{rate:dec:lem} in three
 steps. We begin with a simple lemma that rigorously defines the
 zero-loss and finite-margin sets.
\begin{lemma}
  \label{rate:sol:lem}
  For any sequence $\val_1,\val_2,\ldots,$ of admissible combinations
  of weak classifiers, we can find a subsequence
  $\val_{(1)}=\val_{t_1},\val_{(2)}=\val_{t_2},\ldots,$ whose losses
  converge to zero on all examples in some fixed (possibly empty)
  subset $\zl$ (the zero-loss set), and losses bounded away from zero in
  its complement $X\setminus \zl$(the finite-margin set)
  \begin{eqnarray}
    \label{rate:subseq:eqn}    
    \forall x\in \zl : \lim_{t \to \infty}\loss{\val_{(t)}}{x} = 0,
    &\mbox{  }&
    \forall x\in X\setminus \zl : \inf_i\loss{\val_{(t)}}{x} > 0.
  \end{eqnarray}
\end{lemma}
\begin{proof}
  We will build a zero-loss set and the final subsequence
  incrementally. Initially the set is empty. Pick the first
  example. If the infimal loss ever attained on the example in the
  sequence is bounded away from zero, then we do not add it to the
  set. Otherwise we add it, and consider only the subsequence whose
  $t\th$ element attains loss less than $1/t$ on the
  example. Beginning with this subsequence, we now repeat with other
  examples. The final sequence is the required subsequence, and the
  examples we have added form the zero-loss set.
\end{proof}
We apply Lemma~\ref{rate:sol:lem} to some admissible sequence
converging to the optimal loss (for instance, the one found by AdaBoost). Let
us call the resulting subsequence $\vals_{(t)}$, the obtained
zero-loss set $\zl$, and the finite-margin set $\fl = X\setminus \zl$. The
next lemma shows how to extract a single combination out of the
sequence $\vals_{(t)}$ that satisfies the properties in
Item~\ref{rate:one:dec} of the decomposition lemma.
\begin{lemma}
  \label{rate:basesol:lem}
  Suppose $\M$ is the feature matrix, $\zl$ is a subset of the examples,
  and $\val_{(1)},\val_{(2)},\ldots,$ is a sequence of combinations of
  weak classifiers such that $\zl$ is its zero loss set, and $X\setminus
  \zl$ its finite loss set, that is, \eqref{rate:subseq:eqn} holds. Then
  there is a combination $\alf$ of weak classifiers that achieves
  positive margin on every example in $\zl$, and zero margin on every
  example in its complement $X\setminus \zl$, that is:
  \[
  (\M\alf)_i
  \begin{cases}
    > 0 & \mbox { if } i\in \zl, \\
    = 0 & \mbox { if } i\in X\setminus \zl.
  \end{cases}
  \]
\end{lemma}
\begin{proof}
  Since the $\val_{(t)}$ achieve arbitrarily large positive margins on
  $\zl$, $\norm{\val_{(t)}}$ will be unbounded, and it will be hard to
  extract a useful single solution out of them.  On the other hand,
  the rescaled combinations $\val_{(t)}/\norm{\val_{(t)}}$ lie on a
  compact set, and therefore have a limit point, which might have
  useful properties. We formalize this next.

  We prove the statement of the lemma by induction on the total number
  of training examples $|X|$.  If $X$ is empty, then the lemma holds
  vacuously for any $\alf$. Assume inductively for all $X$ of size
  less than $m>0$, and consider $X$ of size $m$. Since translating a
  vector along the null space of $\M$, $\ker \M = \set{\vx: \M\vx =
    \vzero}$, has no effect on the margins produced by the vector,
  assume without loss of generality that the $\val_{(t)}$'s are
  orthogonal to $\ker \M$. Also, since the margins produced on the
  zero loss set are unbounded, so are the norms of
  $\val_{(t)}$. Therefore assume (by picking a subsequence and
  relabeling if necessary) that $\norm{\val_{(t)}} > t$. Let $\val'$
  be a limit point of the sequence $\val_{(t)} / \norm{\val_{(t)}}$, a
  unit vector that is also orthogonal to the null-space. Then firstly
  $\val'$ achieves non-negative margin on every example; otherwise by
  continuity for some extremely large $t$, the margin of
  $\val_{(t)}/\norm{\val_{(t)}}$ on that example is also negative and
  bounded away from zero, and therefore $\val_{(t)}$'s loss is more
  than $m$, a contradiction to admissibility. Secondly, the margin of
  $\val'$ on each example in $X\setminus \zl$ is zero; otherwise, by
  continuity, for arbitrarily large $t$ the margin of
  $\val_{(t)}/\norm{\val_{(t)}}$ on an example in $X\setminus \zl$ is
  positive and bounded away from zero, and hence that example attains
  arbitrarily small loss in the sequence, a contradiction to
  \eqref{rate:subseq:eqn}. Finally, if $\val'$ achieves zero margin
  everywhere in $\zl$, then $\val'$, being orthogonal to the null-space,
  must be $\vzero$, a contradiction since $\val'$ is a unit
  vector. Therefore $\val'$ must achieve positive margin on some
  non-empty subset $S$ of $\zl$, and zero margins on every other
  example.

  Next we use induction on the reduced set of examples $X'=X\setminus
  S$. Since $S$ is non-empty, $|X'| < m$. Further, using the same
  sequence $\val_{(t)}$, the zero-loss and finite-loss sets,
  restricted to $X'$, are $\zl' = \zl\setminus S$ and $(X\setminus
  \zl)\setminus S =
  X\setminus \zl$ (since $S\subseteq \zl$) $= X'\setminus \zl'$. By the
  inductive hypothesis, there exists some $\val''$ which achieves
  positive margins on $\zl'$, and zero margins on $X'\setminus \zl' =
  X\setminus \zl$. Therefore, by setting $\alf = \val' + c\val''$ for a
  large enough $c$, we can achieve the desired properties.
\end{proof}
Applying Lemma~\ref{rate:basesol:lem} to the sequence $\vals_{(t)}$
yields some convex combination $\alf$ having margin at least $\gamma >
0$ (for some $\gamma$) on $\zl$ and zero margin on its complement,
proving Item~\ref{rate:one:dec} of the decomposition lemma. The next
lemma proves Item~\ref{rate:two:dec}.
\begin{lemma}
  \label{rate:finsolAcomp:lem}
  The optimal loss considering only
  examples within $\fl$ is achieved by some finite combination
  $\vals$.
  % There is a (finite) combination $\vals$ that achieves the same
  % margins on $\fl$ as the optimal solution.
\end{lemma}
\begin{proof}
  The existence of $\alf$ with properties as in
  Lemma~\ref{rate:basesol:lem} implies that the optimal loss is the
  same whether considering all the examples, or just examples in $\fl$.
  Therefore it suffices to show the existence of finite $\vals$ that
  achieves loss $K$ on $\fl$, that is, $\loss{\vals}{\fl}= K$.

  Recall $\M_\fl$ denotes the matrix $\M$ restricted to the rows
  corresponding to examples in $\fl$.  Let $\ker \M_\fl = \set{\vx: \M_\fl\vx
    = 0}$ be the null-space of $\M_\fl$. Let $\val^{(t)}$ be the
  projection of $\vals_{(t)}$ onto the orthogonal subspace of $\ker
  \M_\fl$. Then the losses $\loss{\val^{(t)}}{\fl} = \loss{\vals_{(t)}}{\fl}$
  converge to the optimal loss $K$.  If $\M_\fl$ is identically zero,
  then each $\val^{(t)} = \vzero$, and then $\vals=\vzero$ has loss
  $K$ on $\fl$. Otherwise, let $\lam^2$ be the smallest positive
  eigenvalue of $\M_\fl^T\M_\fl$. Then $\norm{\M\val^{(t)}} \geq
  \lam\norm{\val^{(t)}}$. By the definition of finite margin set,
  $\inf_{t\to\infty}\min_{i\in \fl}\loss{\val^{(t)}}{i} =
  \inf_{t\to\infty}\min_{i\in \fl}\loss{\vals_{(t)}}{i} >
  0$. Therefore, the norms of the margin vectors
  $\norm{\M\val^{(t)}}$, and hence that of $\val^{(t)}$,
  are bounded. Therefore the $\val^{(t)}$'s have a (finite) limit point
  $\vals$ that must have loss $K$ over $\fl$.
\end{proof}
As a corollary, we prove Item~\ref{rate:three:dec}.
\begin{lemma}
  \label{rate:solbound:lem}
  There is a constant $\mu_{\max} < \infty$, such that for any
  combination $\val$ that achieves bounded loss on the finite-margin
  set, $\loss{\val}{\fl} \leq m$, the margin $(\M\val)_i$ for any example
  $i$ in $\fl$ lies in the bounded interval $[-\ln m, \mu_{\max}]$ .
\end{lemma}
\begin{proof}
  Since the loss $\loss{\val}{\fl}$ is at most $m$, therefore no
  margin may be less than $-\ln m$.  To prove a finite upper bound on the
  margins, we argue by contradiction. 
  Suppose  arbitrarily large margins are
  producible by bounded loss vectors, that is arbitrarily large
  elements are present in the set $\set{\enc{\M\val}_i:
    \loss{\val}{\fl}\leq m, 1 \leq i \leq m}$.
  Then for some fixed example $x\in \fl$  
  there exists a sequence of combinations of weak classifiers, whose
  $t\th$ element achieves more than margin $t$ on $x$ but has loss at
  most $m$ on $\fl$. Applying Lemma~\ref{rate:sol:lem} we can find a
  subsequence $\vlam^{(t)}$ whose tail achieves vanishingly small loss on some
  non-empty subset $S$ of $\fl$ containing $x$, and bounded margins in
  $\fl\setminus S$. Applying Lemma~\ref{rate:basesol:lem} to
  $\vlam^{(t)}$ we get some convex combination $\vlamd$ which has
  positive margins on $S$ and zero margin on $\fl \setminus S$. Let
  $\vals$ be as in Lemma~\ref{rate:finsolAcomp:lem}, a finite
  combination achieving the optimal loss on $\fl$. Then $\vals +
  \infty \cdot \vlamd$ achieves the same loss on every example in
  $\fl\setminus S$ as the optimal solution $\vals$, but zero loss for
  examples in $S$. This solution is strictly better than $\vals$ on
  $\fl$, a contradiction to the optimality of $\vals$. Therefore our
  assumption is false, and some finite  upper bound
  $\mu_{\max}$ on the margins $(\M\val)_i$ of vectors satisfying
  $\loss{\val}{\fl} \leq m$ exists.
\end{proof}
}
\longv{
\subsection{Investigating the constants}
\label{rate:constants:sec}

In this section, we try to estimate the constant $C$ in
Theorem~\ref{rate:rate:thm}. We show that it can be arbitrarily large
for adversarial feature matrices with real entries (corresponding to
confidence rated weak hypotheses), but has an upper-bound doubly
exponential in the number of examples when the feature matrix has
$\set{-1,0,+1}$ entries only. We also show that this doubly exponential
bound cannot be improved without significantly changing the proof in
the previous section. 

By inspecting the proofs, we can bound the constant in 
Theorem~\ref{rate:rate:thm} as follows.
 \begin{corollary}
   \label{rate:rate:cor}
   The constant $C$ in Theorem~\ref{rate:rate:thm} that emerges from
   the proofs is
   \[
   C = \frac{32m^3Ne^{\mu_{\max}}}{\gamma^2\lambda^2_{\min}},
   \]
   where $m$ is the number of examples, $N$ is the number of
   hypotheses, $\gamma$ and $\mu_{\max}$ are as given by
   Items~\ref{rate:one:dec} and \ref{rate:three:dec} of the
   decomposition lemma, and $\lmin^2$ is the
   smallest positive eigenvalue of $\M_F^T\M_F$ ($\M_F$ is the feature
   matrix restricted to the rows belonging to the finite margin set
   $F$).
 \end{corollary}
 Our bound on $C$ will be obtained by in turn bounding the quantities
 $\lmin^{-1},\gamma^{-1},\mu_{\max}$. These are strongly related to
 the singular values of the feature matrix $\M$, and in general
 cannot be easily measured. In fact, when $\M$ has real entries, we
 have already seen in Section~\ref{rate:B_lbnd:sec} that the rate can
 be arbitrarily large, implying these parameters can have very large
 values.
 Even when the matrix $\M$ has integer entries
 (that is, $-1, 0, +1$), the next lemma shows that these quantities can be
 exponential in the number of examples.
 \begin{lemma}
   \label{rate:M_int:lem}
   There are examples of feature matrices with $-1,0,+1$ entries and
   at most $m$ rows or columns (where $m>10$) for which the quantities
   $\gamma^{-1}, \lambda^{-1}$ and
   $\mu_{\max}$ are at least $\Omega(2^m/m)$. 
 \end{lemma}
 \begin{proof}
   We first show the bounds for $\gamma$ and $\lambda$. Let $\M$ be
   an $m\times m$ upper triangular matrix with $+1$ on the diagonal,
   and $-1$ above the diagonal. Let $\vy =
   (2^{m-1},2^{m-2},\ldots,1)^T$, and $\vb = (1,1,\ldots,1)^T$. Then
   $\M\vy = \vb$, although the $\vy$ has much bigger norm than $\vb$:
   $\norm{\vy} \geq 2^{m-1}$, while $\norm{\vb} = m$. Since $\M$ is
   invertible, by the definition of $\lmin$, we have $\norm{\M\vy}
   \geq \lmin \norm{\vy}$, so that $\lmin^{-1} \geq
   \norm{\vy}/\norm{\M\vy} \geq 2^m/m$. Next, note that $\vy$ produces
   all positive margins $\vb$, and hence the zero-loss set consists of
   all the examples.  In particular, if $\alf$ be as in
   Item~\ref{rate:one:dec} of the decomposition lemma, then the vector
   $\gamma^{-1}\alf$ achieves more than 1 margin on each example:
   $\M(\gamma^{-1}\alf) \geq \vb$. On the other hand, our matrix is
   very similar to the one in Lemma~\ref{rate:B_big:thm}, and the same
   arguments in the proof of that lemma can be used to show that if
   for some $\vx$ we have $(\M\vx) \geq \vb$, then $\vx \geq \vy$.  This
   implies that $\gamma^{-1}\norm{\alf}_1 \geq \norm{\vy}_1 =
   (2^m-1)$. Since $\alf$ has unit $\ell_1$-norm, the bound on
   $\gamma^{-1}$ follows too.
   
   Next we provide an example showing $\mu_{\max}$ can be
   $\Omega(2^m/m)$. Consider an $m \times (m-1)$ matrix $\M$. The
   bottom row of $\M$ is all $+1$. The upper $(m-1)\times (m-1)$
   submatrix of $\M$ is a lower triangular matrix with $-1$ on the
   diagonal and $+1$ below the diagonal. Observe that if $\vy^T =
   (2^{m-2}, 2^{m-3}, \ldots, 1, 1)$, then $\vy^T\M =
   \vzero$. Therefore, for any vector $\vx$, the inner product of the
   margins $\M\vx$ with $\vy$ is zero: $\vy^TM\vx = 0$. This implies
   that achieving positive margin on any example forces some other
   example to receive negative margin. By Item~\ref{rate:one:dec} of
   the decomposition lemma, the 
   zero loss set in this dataset is empty, and all the examples belong
   to the finite loss set. Next, we choose a combination with at most
   $m$ loss that nevertheless achieves $\Omega(2^m/m)$ positive margin
   on some example. Let $\vx^T = (1,2,4,\ldots,2^{m-2})$. Then
   $(\M\vx)^T = (-1,-1,\ldots,-1,2^{m-1}-1)$. Then the margins using
   $\eps\vx$ are $(-\eps,\ldots,-\eps, \eps(2^{m-1}-1))$ with total
   loss $(m-1)e^\eps + e^{\eps(1-2^{m-1})}$. Choose
   $\eps=1/(2m) \leq 1$, so that the loss on examples corresponding
   to the first $m-1$ rows is at most $e^\eps \leq 1+2\eps = 1+1/m$,
   where the first inequality holds since $\eps \in [0,1]$. For
   $m>10$, the choice of $\eps$ guarantees $1/(2m) = \eps \geq (\ln
   m)/(2^{m-1}-1)$, so that the loss on the example corresponding to
   the bottom most row is $e^{-\eps(2^{m-1}-1)} \leq e^{-\ln m} =
   1/m$. Therefore the net loss of $\eps\vx$ is at most $(m-1)(1+1/m)
   + 1/m = m$. On the other hand the margin on the example
   corresponding to the last row is $\eps(2^{m-1}-1) =
   (2^{m-1}-1)/(2m) = \Omega(2^m/m)$.
 \end{proof}
 The above result implies any bound on $C$ derived from
 Corollary~\ref{rate:rate:cor} will be at least $2^{\Omega(2^{m}/m)}$
 in the worst case. This does not imply that the best bound one can
 hope to prove is doubly exponential, only that our techniques in the
 previous section do not admit anything better. We next show that
 the bounds in Lemma~\ref{rate:M_int:lem} are nearly the worst
 possible.
\begin{lemma}
  \label{rate:const_bnd:lem}
  Suppose each entry of $\M$ is $-1,0$ or $+1$. Then each of the
  quantities $\lmin^{-1}, \gamma^{-1}$ and $\mu_{\max}$ are at most
  $2^{O(m\ln m)}$.
\end{lemma}
The proof of Lemma~\ref{rate:const_bnd:lem} is rather technical, and
we defer it to the Appendix. Lemma~\ref{rate:const_bnd:lem} and
Corollary~\ref{rate:rate:cor} together imply a convergence rate of
$2^{2^{O(m\ln m)}}/\eps$ to the optimal loss for integer
matrices. This bound on $C$ is exponentially worse than the
$\Omega(2^m)$ lower bound on $C$ we saw in
Section~\ref{rate:B_lbnd:sec}, a price we pay for obtaining optimal
dependence on $\eps$. In the next section we will see how to obtain
$\poly(2^{m\ln m},\eps^{-1})$ bounds, although with a worse dependence
on $\eps$. We end this section by showing, just for completeness, how
a bound on the norm of $\vals$ as defined in Item~\ref{rate:two:dec}
of the decomposition lemma follows as a quick corollary to
Lemma~\ref{rate:const_bnd:lem}.
\begin{corollary}
  \label{rate:vals:cor}
  Suppose $\vals$ is as given by Item~\ref{rate:two:dec} of the
  decomposition lemma. When the feature matrix has only $-1,0,+1$
  entries, we may bound $\norm{\vals}_1 \leq 2^{O(m\ln m)}$.
\end{corollary}
\begin{proof}
  Note that every entry of $\M_F\vals$ lies in the range $[-\ln m,
  \mu_{\max}=2^{O(m\ln m)}]$, and hence $\norm{\M_F\vals} \leq
  2^{O(m\ln m)}$. Next, we may choose $\vals$ orthogonal to the null
  space of $\M_F$; then $\norm{\vals} \leq \lmin^{-1}\norm{\M_F\vals}
  \leq 2^{O(m\ln m)}$. Since $\norm{\vals}_1 \leq
  \sqrt{N}\norm{\vals}$, and the number of possible columns $N$ with
  $\set{-1,0,+1}$ entries is at
  most $3^m$, the proof follows.
\end{proof}

\section{Improved Estimates}
\label{rate:improved:sec}

In this section we shed more light on the rate bounds by
cross-application of techniques from Sections~\ref{rate:B:sec} and
\ref{rate:eps:sec}. We obtain both new upper bounds for convergence to
the optimal loss, as well as lower bounds for convergence to an
arbitrary target loss. We also indicate what we believe might be the
optimal bounds for either situation.

We first show how the finite rate bound of
Theorem~\ref{rate:B_rate:thm} along with the decomposition
lemma yields a new rate of convergence to the
optimal loss. Although the dependence on $\eps$ is worse than in
Theorem~\ref{rate:rate:thm}, the dependence on $m$ is nearly
optimal. We will need the following key application of the
decomposition lemma.
\begin{lemma}
  \label{rate:rate_new:lem}
  When the feature matrix has $-1,0,+1$ entries, for any $\eps>0$,
  there is some solution with $\ell_1$-norm at most $2^{O(m\ln
    m)}\ln(1/\eps)$ that achieves within $\eps$ of the optimal loss.
\end{lemma}
\begin{proof}
  Let $\vals,\alf,\gamma$ be as given by the decomposition lemma. Let
  $c = \min_{i\in Z}\enc{\M\vals}_i$ be the minimum margin produced by
  $\vals$ on any example in the zero-loss set $Z$. Then $\vals -
  c\alf$ produces non-negative margins on $Z$, and the optimal margins
  on the finite loss set $F$. Therefore, the vector $\vlams = \vals +
  \enc{\ln(1/\eps)\gamma^{-1}-c}\alf$ achieves at least $\ln(1/\eps)$
  margin on every example in $Z$, and optimal margins on the finite
  loss set $F$. Hence $\lf(\vlams) \leq \inf_{\vlam}\lf(\vlam) +
  \eps$.  Using $\abs{c} \leq \norm{\M\vals} \leq
  m\norm{\vals}$, and the results in Corollary~\ref{rate:vals:cor} and
  Lemma~\ref{rate:const_bnd:lem}, we may conclude the vector $\vlams$ has
  $\ell_1$-norm at most $2^{O(m\ln m)}\ln(1/\eps)$.
\end{proof}
We may now invoke Theorem~\ref{rate:B_rate:thm} to obtain a $2^{O(m\ln
  m)}\ln^6(1/\eps)\eps^{-5}$ rate of convergence to the optimal
solution. Rate bounds with similar dependence on $m$ and slightly
better dependence on $\eps$ can be obtained by modifying the proof in
Section~\ref{rate:eps:sec} to use first order instead of second order
techniques. In that way we may obtain a
$\poly(\lmin^{-1},\gamma^{-1},\mu_{\max})\eps^{-3} = 2^{O(m\ln
  m)}\eps^{-3}$ rate bound. We omit the the rather long but
straightforward proof of this fact. Finally, note that if
Conjecture~\ref{rate:B_rate:conj} is true, then
Lemma~\ref{rate:rate_new:lem} implies a $2^{O(m\ln
  m)}\ln(1/\eps)\eps^{-1}$ rate bound for converging to the optimal
loss, which is nearly optimal in both $m$ and $\eps$. We state this as
an independent conjecture.
\begin{conjecture}
  \label{rate:rate:conj}
  For feature matrices with $-1,0,+1$ entries, AdaBoost converges to
  within $\eps$ of the optimal loss within $2^{O(m\ln
    m)}\eps^{-(1+o(1))}$ rounds.
\end{conjecture}

We next focus on lower bounds on the convergence rate to arbitrary
target losses discussed in Section~\ref{rate:B:sec}. We begin by
showing the rate dependence on the norm of the solution as given in
Lemma~\ref{rate:B_lbnd:lem} holds for much more general datasets.
\begin{lemma}
  \label{rate:B_lbnd_imp:lem}
  Suppose a feature matrix has only $\pm 1$ entries, and the finite
  loss set is non-empty. Then, for any coordinate descent procedure,
  the number of rounds required to achieve a target loss $\phis$ is at
  least
  \[\inf \set{\norm{\vlam}_1:  \lf(\vlam) \leq \phis}/(1+\ln m).\]
\end{lemma}
\begin{proof}
  It suffices to upper-bound the step size $\abs{\alpha_t}$ in any
  round $t$ by at most $1+\ln
  m$. Notice that when the feature matrix has $\pm 1$ entries, a step
  in a direction that does not end up increasing the loss is at most
  of length $(1/2)\ln\enc{\enc{1+\delta}/\enc{1-\delta}}$, where $\delta$ is
  the edge in that direction. Therefore, if $\delta_t$ is the maximum edge
  achievable in any direction, we have
  \[
  \abs{\alpha_t} \leq
  \frac{1}{2}\ln\enc{\frac{1+\delta_t}{1-\delta_t}}.
  \]
  Further, by \eqref{rate:drop:eqn}, a large edge $\delta_t$ ensures
  that for some coordinate step, the new vector $\vlam^t$ will have
  much smaller loss than the vector $\vlam^{t-1}$ at the beginning of
  round $t$: $\lf(\vlam^{t}) \leq
  \lf(\vlam^{t-1})\sqrt{1-\delta_t^2}$.  On the other hand, before the
  step, the loss is at most $1$, $\lf(\vlam^{t-1}) \leq 1$, and after
  the step the loss is at most $1/m$ (since the optimal loss on a
  dataset with non-empty finite set is at least $1/m$): $\lf(\vlam^t)
  \geq 1/m$. Combining these inequalities we get
  \[
  1/m \leq \lf(\vlam^t) \leq \lf(\vlam^{t-1})\sqrt{1-\delta_t^2} \leq
  \sqrt{1-\delta_t^2},
  \]
  that is, $\sqrt{1-\delta_t^2} \geq 1/m$. Now the step length can be bounded as
  \[
  \abs{\alpha_t} \leq
  \frac{1}{2}\ln\enc{\frac{1+\delta_t}{1-\delta_t}} = \ln(1+\delta_t) -
  \frac{1}{2}\ln(1-\delta_t^2) \leq \delta_t + \ln m \leq 1 + \ln m.
  \]
\end{proof}
We end by showing a new lower bound for the convergence rate to an
arbitrary target loss studied in Section~\ref{rate:B:sec}.
Corollary~\ref{rate:B_big:cor} implies that the rate bound in
Theorem~\ref{rate:B_rate:thm} has to be at least polynomially large in
the norm of the solution. We now show that a polynomial dependence on
$\eps^{-1}$ in the rate is unavoidable too. This shows that rates for
competing with a finite solution are different from rates on a
dataset where the optimum loss is achieved by a finite solution, since
in the latter we may achieve a $O\enc{\ln(1/\eps)}$ rate.
\begin{corollary}
  \label{rate:B_dec_lbnd:cor}
  Consider any dataset (e.g. the one in Figure~\ref{rate:dat:fig}) for
  which $\Omega(1/\eps)$ rounds are necessary to get within $\eps$ of
  the optimal loss. If there are constants $c$ and $\beta$ such that
  for any $\vlams$ and $\eps$, a loss of $\lf(\vlams) + \eps$ can be
  achieved in at most $O(\norm{\vlams}_1^c\eps^{-\beta})$ rounds, then
  $\beta \geq 1$.
\end{corollary}
\begin{proof}
  The decomposition lemma implies that $\vlams = \vals +
  \ln(2/\eps)\alf$ with $\ell_1$-norm $O(\ln(1/\eps))$ achieves loss
  at most $K + \eps/2$ (recall $K$ is the optimal loss). Suppose the corollary
  fails to hold for constants $c$ and $\beta \leq 1$. Then $L(\vlams)
  + \eps/2 = K+\eps$ loss can be achieved in
  $O(\eps^{-\beta})/\ln^c(1/\eps)) = o(1/\eps)$ rounds, contradicting
  the $\Omega(1/\eps)$ lower bound.
\end{proof}
% The above result does not say anything about the exponent of the norm
% of the reference solution. The results in
% Section~\ref{rate:B_lbnd:sec} indicates that a dependence on
% $\norm{\vals}_1$ is necessary but does not show how much dependence is
% needed. For instance, Lemma~\ref{rate:B_lbnd:lem} talks about norms of
% solutions achieving a target loss, and not within $\eps$ of the target
% loss, and hence is not directly applicable. We next show a result that
% does imply a polynomial dependence on the norm.
% \begin{corollary}
%   \label{rate:B_norm:cor}
%   Consider the dataset in Figure~\ref{rate:dat:fig}. If there are constants
%   $c$ and $\beta$ such that for any $\vlams$ and $\eps$, a loss of
%   $\lf(\vlams) + \eps$ can be achieved in at most
%   $O(\norm{\vlams}_1^c\eps^{-\beta})$ rounds, then both $\beta$ and
%   $\c$ are at least $1$.
% \end{corollary}
% \begin{proof}
%   That $\beta \geq 1$ follows from
%   Corollary~\ref{rate:B_dec_lbnd:cor}. To show $c\geq 1$, first fix
%   $\eps>0$. Notice that the unique solution achieving exactly $2/3 +
%   \eps$ loss is $(1/2)\ln(1/3\eps)\alf(1,1)$, where
%   $\alf=(1,1)$. Since this dataset satisfies the conditions of
%   Lemma~\ref{rate:B_lbnd:lem}, therefore at least $O(\ln(1/\eps)$
%   rounds are necessary to achieve $2/3 + \eps$ loss. On the other
%   hand, $\norm{\vlams}^c
% \end{proof}

\section{Conclusion}
In this paper we studied the convergence rate of AdaBoost with respect
to the exponential loss. We showed upper and lower bounds for
convergence rates to both an arbitrary target loss achieved by some
finite combination of the weak hypotheses, as well as to the infimum
loss which may not be realizable. For the first convergence rate, we
showed a strong relationship exists between the size of the minimum
vector achieving a target loss and the number of rounds of coordinate
descent required to achieve that loss. In particular, we showed that a
polynomial dependence of the rate on the $\ell_1$-norm $B$ of the
minimum size solution is absolutely necessary, and that a
$\poly(B,1/\eps)$ upper bound holds, where $\eps$ is the accuracy
parameter. The actual rate we derive has rather large exponents, and
we discuss a minor variant of AdaBoost that achieves a much tighter
and near optimal rate.

For the second kind of convergence, using entirely separate techniques,
we derived a $C/\eps$ upper bound, and showed that this is tight up
to constant factors. In the process, we showed a certain decomposition
lemma that might be of independent interest. We also study the
constants and show how they depend on certain intrinsic parameters
related to the singular values of the feature matrix. We estimate the
worst case values of these parameters, and considering feature
matrices with only $\set{-1,0,+1}$ entries, this leads to a bound on
the rate constant $C$ that is doubly exponential in the number of
training examples. Since this is rather large, we also include bounds polynomial
in both the number of training examples and the accuracy parameter
$\eps$, although the dependence on $\eps$ in these bounds is
non-optimal. 

Finally, for each kind of convergence, we conjecture tighter bounds
that are not known to hold presently. A table containing a summary of
the results in this paper is included in Figure~\ref{rate:summary:fig}.

% \begin{tabular}{|c|c|c|c|}
% \hline%
% 1A &\multirow{2}{12mm}{\centering W} & 1 & 2 \tabularnewline
% \cline{1-1} \cline{3-4}%
% 2A & & 2 & 6 \tabularnewline
% \hline%
% 1B& \multirow{2}{12mm}{X} & 3 & 7 \tabularnewline
% \cline{1-1} \cline{3-4}%
% 2B && 4& 8 \tabularnewline
% \hline%
% \end{tabular}

\begin{figure}
  \centering
  \begin{tabular}{|p{3.5cm}|c|c|}
    \hline
    Convergence rate with respect to: &
    Reference solution
    (Section~\ref{rate:B:sec}) &
    Optimal solution
    (Section~\ref{rate:eps:sec})
    \bigstrut \\
    \hline
    \hline
    \multirow{2}{*}
    {Upper bounds:} &
    \multirow{2}{*}
    {${13B^6}/{\eps^5}$} &
    ${\poly(e^{\mu_{\max}},\lmin^{-1},\gamma^{-1})}/{\eps}
    \leq 2^{2^{O(m\ln m)}}/\eps$
    \bigstrut \\
    \cline{3-3}
    &&
    ${\poly(\mu_{\max},\lmin^{-1},\gamma^{-1})}/{\eps^3}
    \leq 2^{O(m\ln m)}/\eps^3$
    \bigstrut \\
    \hline\hline%
    {Lower bounds with:} &
    $(B/\eps)^{1-\nu}$ for any constant $\nu$ &
    \multirow{2}{*}
    {$\max\set{\frac{2^{m}\ln(1/\eps)}{\ln m},
        \frac{2}{9\eps}}$}
    \bigstrut \\
    \cline {2-2}
    a) $\set{0,\pm 1}$ entries  &
    $(2^{m}/\ln m)\ln(1/\eps)$ &
    \bigstrut \\
%     \hline\hline
%     \multirow{2}{*}
%     {Rules out:} &
%     $(B/\eps)^{1-\nu}$ for any constant $\nu$ &
%     \multirow{2}{*}
%     {$\poly(m,\log N, 1/\eps)$}
%     \bigstrut \\
%     \cline{2-2}
%     & $\poly(m,N,1/\eps)$ &
%     \bigstrut \\
    \hline
    b) real entries &
    \multicolumn{2}{c|}
    {Can be arbitrarily large even when $m,N,\eps$ are held fixed}
     \\
    \hline\hline
    Conjectured upper bounds: &
    $O({B^2}/{\eps})$ &
    ${2^{O(m\ln m)}}/{\eps^{1+o(1)}}$, if entries in $\set{0,\pm 1}$
    \bigstrut \\
    \hline
  \end{tabular}
  \caption{Summary of our most important results and conjectures
    regarding the convergence rate of AdaBoost. Here $m$ refers to the
    number of training examples, and $\eps$ is the accuracy
    parameter. The quantity $B$ is the $\ell_1$-norm of the reference
    solution used in Section~\ref{rate:B:sec}. The parameters $\lmin$,
    $\gamma$ and $\mu_{\max}$ depend on the dataset and are defined
    and studied in Section~\ref{rate:eps:sec}.}
  \label{rate:summary:fig}
\end{figure}

\subsection*{Acknowledgments}
This research was funded by the National Science Foundation under
grants IIS-1016029 and IIS-1053407.
We thank Nikhil Srivastava for informing us of the
matrix used in Theorem~\ref{rate:B_big:thm}. We also thank Aditya
Bhaskara and Matus Telgarsky for many helpful discussions.
}
\old{
\newpage
\appendix
\section*{Appendix}
\begin{lemma}
  \label{rate:eps_lbnd:lem} % show dataset requires 0.33 /eps
  For any $\eps<1/3$, to get within $\eps$ of the optimum loss on the
  dataset in Table~\ref{rate:dat:fig}, AdaBoost takes at least
  $2/(9\eps)$ steps. 
\end{lemma}
\begin{proof}
  Note that the optimal loss is $2/3$, and we are bounding the number of
  rounds necessary to get within $(2/3)+\eps$ loss for $\eps < 1/3$.
  We will compute the edge in each round analytically. Let
  $w^t_a,w^t_b,w^t_c$ denote the normalized-losses (adding up to 1) or
  weights on examples $a,b,c$ at the beginning of round $t$,
   $h_t$ the weak hypothesis
  chosen in round $t$, and $\delta_t$ the edge in round $t$. The
  values of these parameters are shown below for the first 5 rounds,
  where we have assumed (without loss of generality) that the
  hypothesis picked in round 1 is $\bh_b$:
  \begin{center}
    \begin{tabular}{c|ccc|c|c}
     Round  & $w^t_a$ & $w^t_b$ & $w^t_c$ & $h_t$ & $\delta_t$ \\
      \hline 
      $t=1:$ & $1/3$ & $1/3$ & $1/3$ & $\bh_b$ & $1/3$ \\
      $t=2:$ & $1/2$ & $1/4$ & $1/4$ & $\bh_a$ & $1/2$ \\
      $t=3:$ & $1/3$ & $1/2$ & $1/6$ & $\bh_b$ & $1/3$ \\
      $t=4:$ & $1/2$ & $3/8$ & $1/8$ & $\bh_a$ & $1/4$ \\
      $t=5:$ & $2/5$ & $1/2$ & $1/10$ & $\bh_b$ & $1/5$.
    \end{tabular}
  \end{center}
  Based on the patterns above, we first claim that for rounds $t \geq
  2$,  the edge achieved  is $1/t$. 
  In fact we prove the stronger claims, that for
  rounds $t\geq 2$, the following hold:
   \begin{enumerate}
  \item One of $w^t_a$ and $w^t_b$ is $1/2$.
  \item $\delta_{t+1} = \delta_t/(1+\delta_t)$. 
  \end{enumerate}
  Since $\delta_2 = 1/2$, the recurrence on $\delta_t$ would
  immediately imply $\delta_t = 1/t$ for $t\geq 2$.  We prove the
  stronger claims by induction on the round $t$.  The base case for
  $t=2$ is shown above and may be verified.  Suppose the inductive
  assumption holds for $t$.  Assume without loss of generality that
  $1/2=w^t_a > w^t_b > w^t_c$; note this implies $w^t_b =
  1-(w^t_a+w^t_c) = 1/2-w^t_c$.  Further, in this round, $\bh_a$ gets
  picked, and has edge $\delta_t = w^t_a + w^t_c - w^t_b = 2w^t_c$.
  Now for any dataset, the weights of the examples labeled correctly and
  incorrectly in a round of AdaBoost are rescaled during the weight
  update step in a way such that each add up to $1/2$ after the rescaling.
  Therefore, $w^{t+1}_b = 1/2, w^{t+1}_c =
  w^t_c\enc{\frac{1/2}{w^t_a + w^t_c}} = w^t_c/(1+2w^t_c)$.  Hence,
  $\bh_b$ gets picked in round $t+1$ and, as before, we get edge
  $\delta_{t+1} = 2w^{t+1}_c = 2w^t_c/(1+2w^t_c) =
  \delta_t/(1+\delta_t)$. The proof of our claim follows by induction.

  Next we find the loss after each iteration. Using $\delta_1 = 1/3$
  and $\delta_t = 1/t$ for $t\geq 2$, the loss after $T$ rounds can be
  written as
  \[
  \prod_{t=1}^T\sqrt{1-\delta_t^2} =
  \sqrt{1- (1/3)^2}\prod_{t=2}^T \sqrt{1-1/t^2}
  = \frac{2\sqrt{2}}{3}\sqrt{\prod_{t=2}^T
    \enc{\frac{t-1}{t}}
    \enc{\frac{t+1}{t}}}.
  \]
  The product can be rewritten as follows:
  \[
  \prod_{t=2}^T\enc{\frac{t-1}{t}}\enc{\frac{t+1}{t}}
  =\enc{\prod_{t=2}^T\frac{t-1}{t}}
  \enc{\prod_{t=2}^T\frac{t+1}{t}}
  =\enc{\prod_{t=2}^T\frac{t-1}{t}}
  \enc{\prod_{t=3}^{T+1}\frac{t}{t-1}}.
  \]
  Notice almost all the terms cancel, except for the first term of the
  first product, and the last term of the second product.
  Therefore, the loss after $T$ rounds is
  \[
  \frac{2\sqrt{2}}{3}
  \sqrt{\enc{\frac{1}{2}}
    \enc{\frac{T+1}{T}}}
  = \frac{2}{3}\sqrt{1+\frac{1}{T}}
  \geq \frac{2}{3}\enc{1 +
    \frac{1}{3T}}
  = \frac{2}{3} + \frac{2}{9T},
  \]
  where the inequality holds for $T\geq 1$.
  Since the
  initial error is $1 = (2/3) + 1/3$, therefore, for any $\eps < 1/3$,
  the number of rounds needed to achieve loss $(2/3) + \eps$ is at
  least $2/(9\eps)$.
\end{proof}
\begin{lemma}
  \label{rate:rec:lem}
  Suppose $u_0, u_1, \ldots, $ are non-negative numbers satisfying
  \[
  u_t - u_{t+1} \geq c_0u_t^{1+c_1},
  \]
  for some non-negative constants $c_0,c_1$. Then, for any $t$,
  \[
  \frac{1}{u_t^{c_1}} - \frac{1}{u_0^{c_1}} \geq c_1c_0t.
  \]
\end{lemma}
\begin{proof}
  By induction on $t$. The base case is an identity. Assume the statement
  holds at iteration $t$. Then,
  \begin{eqnarray*}
    \frac{1}{u_{t+1}^{c_1}} - \frac{1}{u_0^{c_1}} =
    \enc{\frac{1}{u_{t+1}^{c_1}} - \frac{1}{u_t^{c_1}}} + 
    \enc{\frac{1}{u_t^{c_1}} - \frac{1}{u_0^{c_1}}} 
    \geq \frac{1}{u_{t+1}^{c_1}} - \frac{1}{u_t^{c_1}} + c_1c_0t \mbox{
      (by inductive hypothesis). }
  \end{eqnarray*}
  Thus it suffices to show
  $1/u_{t+1}^{c_1} - 1/u_t^{c_1} \geq c_1c_0$. Multiplying both sides
  by $u_t^{c_1}$ and adding 1, this is equivalent to showing
  $(u_t/u_{t+1})^{c_1} \geq 1 + c_1c_0u_t^{c_1}$.
  We will in fact show the stronger inequality
  \begin{equation}
    \label{rate:rec_one:eqn}
  \enc{u_t/u_{t+1}}^{c_1} \geq \enc{1 + c_0u_t^{c_1}}^{c_1}.
  \end{equation}
  Since $(1+a)^b \geq   1+ba$ for $a,b$ non-negative,
  \eqref{rate:rec_one:eqn} will imply $\enc{u_t/u_{t+1}}^{c_1} \geq
  \enc{1 + c_0u_t^{c_1}}^{c_1} \geq 1 + c_1c_0u_t^{c_1}$, which will
  complete our proof.
  To show \eqref{rate:rec_one:eqn}, we first rearrange the condition on
  $u_t,u_{t+1}$ to obtain
  \[
  u_{t+1} \leq u_t\enc{1-c_0u_t^{c_1}} \implies
  \frac{u_t}{u_{t+1}} \geq \frac{1}{1-c_0u_t^{c_1}}.
  \]
  Applying the fact $\enc{1+c_0u_t^{c_1}}\enc{1-c_0u_t^{c_1}} \leq 1$
  to the previous equation we get,
  \[
  \frac{u_t}{u_{t+1}} \geq 1+c_0u_t^{c_1}.
  \]
  Since $c_1 \geq 0$, we may raise both sides of the above inequality
  to the power of $c_1$ to show \eqref{rate:rec_one:eqn}, finishing
  our proof.
\end{proof}
}
\longv{
\subsection*{Proof of Lemma~\ref{rate:const_bnd:lem}}

In this section we prove Lemma~\ref{rate:const_bnd:lem}, by separately
bounding the quantities $\lmin^{-1}$, $\gamma^{-1}$ and $\mu_{\max}$,
through a sequence of Lemmas.  We will use the next result repeatedly.
\begin{lemma}
  \label{rate:inv:lem}
  If $\A$ is an $n\times n$ invertible matrix with $-1,0,+1$ entries,
  then $\min_{\vx: \norm{\vx} = 1} \norm{\A\vx}$ is at least $1/n! =
  2^{-O(n\ln n)}$.
\end{lemma}
\begin{proof}
  It suffices to show that $\norm{\A^{-1}\vx} \leq n!$ for any $\vx$
  with unit norm. Now $\A^{-1} = \rm{adj}(\A)/\det(\A)$ where
  $\rm{adj}(\A)$ is the adjoint of $\A$, whose $i,j$-th entry is the
  $i,j$th cofactor of $\A$ (given by $(-1)^{i+j}$ times the determinant
  of the $n-1 \times n-1$ matrix obtained by removing the $i$th row
  and $j$th column of $\A$), and $\det(\A)$ is the determinant of
  $\A$. The determinant of any $k\times k$ matrix $G$ can be written as
  $\sum_{\sigma}\sgn(\sigma)\prod_{i=1}^kG(i,\sigma(j))$, where
  $\sigma$ ranges over all the permutations of $1,\ldots,
  k$. Therefore each entry of $\rm{adj}(\A)$ is at most $(n-1)!$, and
  the $\det(\A)$ is a non-zero integer. Therefore $\norm{\A^{-1}\vx} =
  \norm{\rm{adj}(\A)\vx}/\det(\A) \leq n!\norm{\vx}$, and the proof is
  complete.
\end{proof}
We first show our bound holds for $\lmin$.
\begin{lemma}
  \label{rate:lmin:lem}
  Suppose $\M$ has $-1,0,+1$ entries, and let $\M_F,\lmin$ be as in
  Corollary~\ref{rate:rate:cor}. Then $\lmin \geq 1/m!$.
\end{lemma}
\begin{proof}
  Let $\A$ denote the matrix $\M_F$. It suffices to show that $\A$ does
  not squeeze too much the norm of any vector orthogonal to the
  null-space $\ker \A \eqdef \set{\val:\A\val = \vzero}$ of $\A$, i.e. $
  \norm{\A\vlam} \geq (1/m!)\norm{\vlam}$ for any $\vlam \in \ker
  \A^{\bot}$. We first characterize $\ker \A^{\bot}$ and then study how
  $\A$ acts on this subspace.

  Let the rank of $\A$ be $k\leq m$ (notice $\A=\M_F$ has $N$ columns and
  fewer than $m$ rows). Without loss of generality, assume the first
  $k$ columns of $\A$ are independent. Then every column of $\A$ can be
  written as a linear combination of the first $k$ columns of $\A$, and
  we have $ \A = \A' [ \vI | \B ]$ (that is, the matrix $\A$ is the product
  of matrices $\A'$ and $[\vI|\B]$), where $\A'$ is the submatrix
  consisting of the first $k$ columns of $\A$, $\vI$ is the $k \times k$
  identity matrix, and $\B$ is some $k \times (N-k)$ matrix of linear
  combinations (here $|$ denotes concatenation). The null-space of $\A$
  consists of $\vx$ such that $\vzero = \A\vx = \A'[\vI | \B] \vx =
  \A'(\vx_k + \B\vx_{-k})$, where $\vx_k$ is the first $k$ coordinates
  of $\vx$, and $\vx_{-k}$ the remaining $N-k$ coordinates. Since the
  columns of $\A'$ are independent, this happens if and only if $\vx_k
  = -\B\vx_{-k}$. Therefore $\ker \A = \set{(-\B\vz,\vz):
    \vz\in\R^{N-k}}$. Since a vector $\vx$ lies in the orthogonal
  subspace of $\ker \A$ if it is orthogonal to every vector in the
  latter, we have
  \[
  \ker \A^{\bot} =  \set{(\vx_k,\vx_{-k}):
    \dotp{\vx_k}{\B\vz} = \dotp{\vx_{-k}}{\vz},
    \forall \vz\in \R^{N-K}}.
  \]
  We next see how $\A$ acts on this subspace. Recall $\A = \A'[\vI|\B]$
  where $\A'$ has $k$ independent columns. By basic linear algebra, the
  row rank of $\A'$ is also $k$, and assume without loss of generality
  that the first $k$ rows of $\A'$ are independent. Denote by $\A_k$ the
  $k\times k$ submatrix of $\A'$ formed by these $k$ rows. Then for any
  vector $\vx$,
  \[
  \norm{\A\vx} = \norm{\A'[\vI|\B]\vx} = \norm{\A'(\vx_k + \B\vx_{-k})} \geq
  \norm{\A_k(\vx_k + \B\vx_{-k})} \geq \frac{1}{k!}\norm{\vx_k + \B\vx_{-k}},
  \]
  where the last inequality follows from Lemma~\ref{rate:inv:lem}. To
  finish the proof, it suffices to show that $\norm{\vx_k + \B\vx_{-k}}
  \geq \norm{\vx}$ for $\vx\in\ker \A^{\bot}$. Indeed, by expanding out
  $\norm{\vx_k + \B\vx_{-k}}^2$ as inner product with itself, we have
  \[
  \norm{\vx_k + \B\vx_{-k}}^2 = \norm{\vx_k}^2 + \norm{\B\vx_{-k}}^2 +
  2\dotp{\vx_k}{\B\vx_{-k}} \geq \norm{\vx_k}^2 + 2\norm{\vx_{-k}}^2
  \geq \norm{\vx}^2,
  \]
  where the first inequality follows since $\vx\in\ker \A^{\bot}$
  implies $\dotp{\vx_k}{\B\vx_{-k}} = \dotp{\vx_{-k}}{\vx_{-k}}$.
\end{proof}
To show the bounds on $\gamma^{-1}$ and $\mu_{\max}$, we will need an
intermediate result.
\begin{lemma}
  \label{rate:ineq:lem}
  Suppose $\A$ is a matrix, and $\vb$ a vector, both with $-1,0,1$
  entries. If $\A\vx = \vb, \vx \geq \vzero$ is solvable, then there is
  a solution satisfying $\norm{\vx} \leq k\cdot k!$, where
  $k=\rm{rank}(\A)$. 
\end{lemma}
\begin{proof}
  Pick a solution $\vx$ with maximum number of zeroes. Let $J$ be the
  set of coordinates for which $x_i$ is zero. We first claim that
  there is no other solution $\vx'$ which is also zero on the set
  $J$. Suppose there were such an $\vx'$. Note any point $\vp$ on the
  infinite line joining $\vx,\vx'$ satisfies $\A\vp = \vb$, and $\vp_J
  = \vzero$ (that is, $p_{i'}=0$ for $i'\in J$). If $i$ is any
  coordinate not in $J$ such that $x_i \neq x'_i$, then for some point
  $\vp^i$ along the line, we have $\vp^i_{J\cup\set{i}} =
  \vzero$. Choose $i$ so that $\vp^i$ is as close to $\vx$ as
  possible. Since $\vx \geq \vzero$, by continuity this would also
  imply that $\vp^i \geq \vzero$. But then $\vp^i$ is a solution with
  more zeroes than $\vx$, a contradiction.

  The claim implies that the reduced problem $\A'\tilde{\vx} = \vb,
  \tilde{\vx} \geq \vzero$, obtained by substituting $\vx_J = \vzero$,
  has a unique solution. Let $k=\rm{rank}(\A')$, $\A_k$ be a $k\times k$
  submatrix of $\A'$ with full rank, and $\vb_k$ be the restriction of
  $\vb$ to the rows corresponding to those of $\A_k$ (note that $\A'$,
  and hence $\A_k$, contain only $-1,0,+1$ entries). Then,
  $\A_k\tilde{\vx} = \vb_k, \tilde{\vx} \geq \vzero$ is equivalent to
  the reduced problem. In particular, by uniqueness, solving
  $\A_k\tilde{\vx} = \vb_k$ automatically ensures the obtained $\vx =
  (\tilde{\vx}, \vzero_J)$ is a non-negative solution to the original
  problem, and satisfies $\norm{\vx} = \norm{\tilde{\vx}}$. But, by
  Lemma~\ref{rate:inv:lem},
  \[
  \norm{\tilde{\vx}} \leq k!\norm{\A_k\tilde{\vx}} = k!\norm{\vb_k}
  \leq k\cdot k!.
  \]
\end{proof}
The bound on $\gamma^{-1}$ follows easily.
\begin{lemma}
  \label{rate:gam:lem}
  Let $\gamma,\alf$ be as in Item~\ref{rate:one:dec} of
  Lemma~\ref{rate:dec:lem}. Then $\alf$ can be chosen such that
  $\gamma \geq 1/\enc{\sqrt{N}m\cdot m!} \geq 2^{-O(m\ln m)}$.
\end{lemma}
\begin{proof}
  We know that $\M(\alf/\gamma) = \vb$, where $\vb$ is zero on the set
  $\fl$ and at least $1$ for every example in the zero loss set $\zl$ (as
  given by Item~\ref{rate:one:dec} of Lemma~\ref{rate:dec:lem}). Since
  $\M$ is closed under complementing columns, we may assume in addition
  that $\alf \geq \vzero$. Introduce slack variables $z_i$ for $i\in
  \zl$, and let $\tilde{\M}$ be $\M$ augmented with the columns $-\ve_i$
  for $i\in \zl$, where $\ve_i$ is the standard basis vector with $1$ on
  the $i$th coordinate and zero everywhere else. Then, by setting $\vz
  = \M(\alf/\gamma) - \vb$, we have a solution $(\alf/\gamma,\vz)$ to
  the system $\tilde{\M}\vx = \vb, \vx \geq \vzero$. Applying
  Lemma~\ref{rate:ineq:lem}, we know there exists some solution
  $(\vy,\vz')$ with norm at most $m\cdot m!$ (here $\vz'$ corresponds
  to the slack variables). Observe that $\vy/\norm{\vy}_1$ is a valid
  choice for $\alf$ yielding a $\gamma$ of $1/\norm{\vy}_1 \geq
  1/(\sqrt{N}m\cdot m!)$.
\end{proof}
To show the bound for $\mu_{\max}$ we will need a version of
Lemma~\ref{rate:ineq:lem} with strict inequality.
\begin{corollary}
  \label{rate:ineq:cor}
  Suppose $\A$ is a matrix, and $\vb$ a vector, both with $-1,0,1$
  entries. If $\A\vx = \vb, \vx > \vzero$ is solvable, then there is
  a solution satisfying $\norm{\vx} \leq 1+k\cdot k!$, where
  $k=\rm{rank}(\A)$.   
\end{corollary}
\begin{proof}
  Using Lemma~\ref{rate:ineq:lem}, pick a solution to $\A\vx = \vb, \vx
  \geq \vzero$ with norm at most $k\cdot k!$. If $\vx > \vzero$, then
  we are done. Otherwise let $\vy > \vzero$ satisfy $\A\vx = \vb$, and
  consider the segment joining $\vx$ and $\vy$. Every point $\vp$ on
  the segment satisfies $\A\vp = b$. Further any coordinate becomes
  zero at most once on the segment. Therefore, there are points
  arbitrarily close to $\vx$ on the segment with positive coordinates
  that satisfy the equation, and these have norms
  approaching that of $\vx$.
\end{proof}
We next characterize the feature matrix $\M_\fl$ restricted to the
finite-loss examples, which might be of independent interest.
\begin{lemma}
  \label{rate:sep:lem}
  If $\M_\fl$ is the feature matrix restricted to the finite-loss
  examples $\fl$ (as given by Item~\ref{rate:two:dec} of
  Lemma~\ref{rate:dec:lem}), then there exists a positive linear
  combination $\vy > \vzero$ such that $\M_\fl^T\vy = \vzero$.
\end{lemma}
\begin{proof}
  Item~\ref{rate:three:dec} of the decomposition lemma states that
  whenever the loss $\loss{\vx}{\fl}$ of a vector is bounded by $m$,
  then the largest margin $\max_{i\in \fl} (\M_\fl\vx)_i$ is at most
  $\mu_{\max}$. This implies that there is no vector $\vx$ such that
  $\M_\fl\vx \geq \vzero$ and at least one of the margins $(\M_\fl\vx)_i$ is
  positive; otherwise, an arbitrarily large multiple of $\vx$ would
  still have loss at most $m$, but margin exceeding the constant
  $\mu_{\max}$. In other words, $\M_\fl\vx \geq \vzero$ implies $\M_\fl\vx =
  \vzero$. In particular, the subspace of possible margin vectors
  $\set{\M_\fl\vx: \vx\in \R^N}$ is disjoint from the convex set
  $\Delta_\fl$ of distributions over examples in $\fl$, which consists of
  points in $\R^{|\fl|}$ with all non-negative and at least one positive
  coordinates. By the Hahn-Banach Separation theorem, there exists a
  hyperplane separating these two bodies, i.e. there is a $\vy\in
  \R^{|\fl|}$, such that for any $\vx\in\R^N$ and $\vp\in\Delta_\fl$, we
  have $\dotp{\vy}{\M_\fl\vx} \leq 0 < \dotp{\vy}{\vp}$. By choosing $\vp
  = \ve_i$ for various $i\in \fl$, the second inequality yields $\vy >
  \vzero$. Since $\M_\fl\vx = -\M_\fl(-\vx)$, the first inequality implies
  that equality holds for all $\vx$, i.e. $\vy^T\M_\fl = \vzero^T$.
\end{proof}
We can finally upper-bound $\mu_{\max}$.
\begin{lemma}
  \label{rate:mumax:lem}
  Let $\fl, \mu_{\max}$ be as in
  Items~\ref{rate:two:dec},\ref{rate:three:dec} of the decomposition
  lemma. Then $\mu_{\max} \leq \ln m \cdot
  |\fl|^{1.5}\cdot |\fl|! \leq 2^{O(m\ln m)}$.
\end{lemma}
\begin{proof}
  Pick any example $i\in \fl$ and any combination $\vlam$ whose loss on
  $\fl$, $\sum_{i\in \fl}e^{-(\M\vlam)_i}$, is at most $m$. Let $\vb$ be
  the $i\th$ row of $\M$, and let $\A^T$ be the matrix $\M_\fl$ without the
  $i$th row. Then Lemma~\ref{rate:sep:lem} says that $\A\vy = -\vb$ for
  some positive vector $\vy > \vzero$. This implies the margin of
  $\vlam$ on example $i$ is $(\M\vlam)_i = -\vy^T\A^T\vlam$. Since the
  loss of $\vlam$ on $\fl$ is at most $m$, each margin on $\fl$ is at
  least $-\ln m$, and therefore $\max_{i\in \fl}\enc{-\A^T\vlam}_i \leq
  \ln m$. Hence, the margin on example $i$ can be bounded as
  $(\M\vlam)_i = \dotp{\vy^T}{-\A^T\vlam} \leq \ln m
  \norm{\vy}_1$. Using Corollary~\ref{rate:ineq:cor}, we can find
  $\vy$ with bounded norm,  $\norm{\vy}_1 \leq \sqrt{|\fl|}\norm{\vy}
  \leq \sqrt{|\fl|}(1 + k\cdot k!)$ , where $k = {\rm rank}(\A) \leq
  {\rm rank}(\M_\fl) \leq |\fl|$. The proof follows.
\end{proof}

}

\bibliography{newbib}

\begin{thebibliography}{22}
\providecommand{\natexlab}[1]{#1}
\providecommand{\url}[1]{\texttt{#1}}
\expandafter\ifx\csname urlstyle\endcsname\relax
  \providecommand{\doi}[1]{doi: #1}\else
  \providecommand{\doi}{doi: \begingroup \urlstyle{rm}\Url}\fi

\bibitem[Bartlett and Traskin(2007)]{BartlettTr07}
Peter~L. Bartlett and Mikhail Traskin.
\newblock {AdaBoost} is consistent.
\newblock \emph{Journal of Machine Learning Research}, 8:\penalty0 2347--2368,
  2007.

\bibitem[Bickel et~al.(2006)Bickel, Ritov, and Zakai]{BickelRiZa06}
Peter~J. Bickel, Ya'acov Ritov, and Alon Zakai.
\newblock Some theory for generalized boosting algorithms.
\newblock \emph{Journal of Machine Learning Research}, 7:\penalty0 705--732,
  2006.

\bibitem[Boyd and Vandenberghe(2004)]{BoydVa04}
Stephen Boyd and Lieven Vandenberghe.
\newblock \emph{Convex Optimization}.
\newblock Cambridge University Press, 2004.

\bibitem[Breiman(1999)]{Breiman99}
Leo Breiman.
\newblock Prediction games and arcing classifiers.
\newblock \emph{Neural Computation}, 11\penalty0 (7):\penalty0 1493--1517,
  1999.

\bibitem[Collins et~al.(2002)Collins, Schapire, and Singer]{CollinsScSi02}
Michael Collins, Robert~E. Schapire, and Yoram Singer.
\newblock Logistic regression, {AdaBoost} and {Bregman} distances.
\newblock \emph{Machine Learning}, 48\penalty0 (1/2/3), 2002.

\bibitem[Frean and Downs(1998)]{FreanDo98}
Marcus Frean and Tom Downs.
\newblock A simple cost function for boosting.
\newblock Technical report, Department of Computer Science and Electrical
  Engineering, University of Queensland, 1998.

\bibitem[Freund and Schapire(1997)]{FreundSc97}
Yoav Freund and Robert~E. Schapire.
\newblock A decision-theoretic generalization of on-line learning and an
  application to boosting.
\newblock \emph{Journal of Computer and System Sciences}, 55\penalty0
  (1):\penalty0 119--139, August 1997.

\bibitem[Friedman et~al.(2000)Friedman, Hastie, and Tibshirani]{FriedmanHaTi00}
Jerome Friedman, Trevor Hastie, and Robert Tibshirani.
\newblock Additive logistic regression: {A} statistical view of boosting.
\newblock \emph{Annals of Statistics}, 28\penalty0 (2):\penalty0 337--374,
  April 2000.

\bibitem[Friedman(2001)]{Friedman01}
Jerome~H. Friedman.
\newblock Greedy function approximation: {A} gradient boosting machine.
\newblock \emph{Annals of Statistics}, 29\penalty0 (5), October 2001.

\bibitem[Luenberger and Ye(2008)]{LuenbergerYe08}
David~G. Luenberger and Yinyu Ye.
\newblock \emph{Linear and nonlinear programming}.
\newblock Springer, third edition, 2008.

\bibitem[Luo and Tseng(1992)]{LuoTs92}
Z.~Q. Luo and P.~Tseng.
\newblock On the convergence of the coordinate descent method for convex
  differentiable minimization.
\newblock \emph{Journal of Optimization Theory and Applications}, 72\penalty0
  (1):\penalty0 7--35, January 1992.

\bibitem[Mason et~al.(2000)Mason, Baxter, Bartlett, and Frean]{MasonBaBaFr00}
Llew Mason, Jonathan Baxter, Peter Bartlett, and Marcus Frean.
\newblock Boosting algorithms as gradient descent.
\newblock In \emph{Advances in Neural Information Processing Systems 12}, 2000.

\bibitem[Onoda et~al.(1998)Onoda, R\"atsch, and M\"uller]{OnodaRaMu98}
T.~Onoda, G.~R\"atsch, and K.-R. M\"uller.
\newblock An asymptotic analysis of {AdaBoost} in the binary classification
  case.
\newblock In \emph{Proceedings of the 8th International Conference on
  Artificial Neural Networks}, pages 195--200, 1998.

\bibitem[R\"atsch et~al.(2001)R\"atsch, Onoda, and M\"uller]{RatschOnMu01}
G.~R\"atsch, T.~Onoda, and K.-R. M\"uller.
\newblock Soft margins for {AdaBoost}.
\newblock \emph{Machine Learning}, 42\penalty0 (3):\penalty0 287--320, 2001.

\bibitem[R\"atsch and Warmuth(2005)]{RatschWa05}
Gunnar R\"atsch and Manfred~K. Warmuth.
\newblock Efficient margin maximizing with boosting.
\newblock \emph{Journal of Machine Learning Research}, 6:\penalty0 2131--2152,
  2005.

\bibitem[R\"atsch et~al.(2002)R\"atsch, Mika, and Warmuth]{RatschMiWa02}
Gunnar R\"atsch, Sebastian Mika, and Manfred~K. Warmuth.
\newblock On the convergence of leveraging.
\newblock In \emph{Advances in Neural Information Processing Systems 14}, 2002.

\bibitem[Rudin et~al.(2007)Rudin, Schapire, and Daubechies]{RudinScDa07}
Cynthia Rudin, Robert~E. Schapire, and Ingrid Daubechies.
\newblock Analysis of boosting algorithms using the smooth margin function.
\newblock \emph{Annals of Statistics}, 35\penalty0 (6):\penalty0 2723--2768,
  2007.

\bibitem[Schapire(2010)]{Schapire10}
Robert~E. Schapire.
\newblock The convergence rate of {A}da{B}oost.
\newblock In \emph{The 23rd Conference on Learning Theory}, 2010.
\newblock open problem.

\bibitem[Schapire and Singer(1999)]{SchapireSi99}
Robert~E. Schapire and Yoram Singer.
\newblock Improved boosting algorithms using confidence-rated predictions.
\newblock \emph{Machine Learning}, 37\penalty0 (3):\penalty0 297--336, December
  1999.

\bibitem[Shalev-Shwartz and Singer(2008)]{ShalevshwartzSi08}
Shai Shalev-Shwartz and Yoram Singer.
\newblock On the equivalence of weak learnability and linear separability: New
  relaxations and efficient boosting algorithms.
\newblock In \emph{21st Annual Conference on Learning Theory}, 2008.

\bibitem[Telgarsky(2011)]{Telgarsky11}
Matus Telgarsky.
\newblock The convergence rate of {A}da{B}oost and friends.
\newblock http://arxiv.org/abs/1101.4752, January 2011.

\bibitem[Zhang and Yu(2005)]{ZhangYu05}
Tong Zhang and Bin Yu.
\newblock Boosting with early stopping: Convergence and consistency.
\newblock \emph{Annals of Statistics}, 33\penalty0 (4):\penalty0 1538--1579,
  2005.

\end{thebibliography}

\end{document}